\documentclass{amsart}
\usepackage{latexsym,amscd,array}
\usepackage{enumerate,paralist}
\usepackage{exscale}
\usepackage[centertags]{amsmath}
\usepackage{amssymb,stmaryrd}
\usepackage{amsthm}
\usepackage{dsfont}
\usepackage[all]{xy}
\usepackage{lscape}
\usepackage{pdflscape}
\usepackage{color}
\usepackage{graphicx}
\usepackage{hyperref}
\usepackage{lineno}
\usepackage{mathdots}

\newtheorem{thm}{Theorem}[section]
\newtheorem{prop}[thm]{Proposition}
\newtheorem{lem}[thm]{Lemma}
\newtheorem{cor}[thm]{Corollary}
   \theoremstyle{definition}
\newtheorem{dfn}[thm]{Definition}
\newtheorem{exa}[thm]{Example}
\newtheorem{rmk}[thm]{Remark}

\newtheorem{ntn}[thm]{Notation}
\newtheorem{prb}[thm]{Problem}

\newcommand{\posP}{{\mathsf P}}

\def\schluss{\hfill\ensuremath{\diamond}}

\newcommand{\Ext}{\operatorname{Ext}}

\newcommand{\cs}{{\mathrm{c.s.}}}

\newcommand{\Var}{{\mathrm{Var}}}

\newcommand{\dR}{{\mathrm{dR}}}

\newcommand{\Tor}{\operatorname{Tor}}

\newlength{\myl}
\newcommand{\mmu}{{\makebox[1.17\myl]{$\mu\hspace{-0.83\myl}\mu$}}}

\newcommand{\pt}{\mathit{pt}}
\newcommand{\an}{\mathit{an}}

\newcommand{\del}{\partial}

\newcommand{\de}{{\mathrm d}}

\newcommand{\into}{\hookrightarrow}
\newcommand{\onto}{\twoheadrightarrow}
\renewcommand{\to}{\longrightarrow}
\newcommand{\ideal}[1]{{\langle#1\rangle}}
\newcommand{\ilim}{\varprojlim}

\newcommand{\minus}{\smallsetminus}
\newcommand{\tildeV}{{\tilde V}}

\newcommand{\calD}{\mathcal{D}}

\newcommand{\calF}{\mathcal{F}}
\newcommand{\calG}{\mathcal{G}}

\newcommand{\calH}{\mathcal{H}}

\newcommand{\calL}{\mathcal{L}}
\newcommand{\calM}{\mathcal{M}}

\newcommand{\calO}{\mathcal{O}}

\newcommand{\frakm}{{\mathfrak{m}}}
\newcommand{\frakp}{{\mathfrak{p}}}
\newcommand{\frakn}{{\mathfrak{n}}}

\newcommand{\fraky}{{\mathfrak{y}}}

\renewcommand{\AA}{\mathbb{A}}
\newcommand{\CC}{\mathbb{C}}
\newcommand{\DD}{\mathbb{D}}

\newcommand{\HH}{\mathbb{H}}
\newcommand{\KK}{\mathbb{K}}

\newcommand{\NN}{\mathbb{N}}
\newcommand{\PP}{\mathbb{P}}
\newcommand{\QQ}{\mathbb{Q}}
\newcommand{\RR}{\mathbb{R}}
\renewcommand{\SS}{\mathbb{S}}
\newcommand{\ZZ}{\mathbb{Z}}

\newcommand{\hatOmega}{{\hat \Omega}}

\DeclareMathOperator{\charac}{\textup{char}}
\DeclareMathOperator{\Cl}{\textup{Cl}}
\DeclareMathOperator{\coker}{\textup{coker}}
\DeclareMathOperator{\codim}{\textup{codim}}

\DeclareMathOperator{\DR}{\mathit{DR}}

\DeclareMathOperator{\gr}{\textup{gr}}
\DeclareMathOperator{\Hom}{\textup{Hom}}

\DeclareMathOperator{\id}{\textup{id}}

\DeclareMathOperator{\injdim}{\textup{injdim}}

\DeclareMathOperator{\depth}{\textup{depth}}

\DeclareMathOperator{\Pic}{\textup{Pic}}
\DeclareMathOperator{\Proj}{\textup{Proj}\,}

\DeclareMathOperator{\Supp}{\textup{Supp}}
\DeclareMathOperator{\Spec}{\textup{Spec}\,}

\numberwithin{equation}{subsection}

\begin{document}
\title{On Lyubeznik type invariants}

\author{Thomas Reichelt}
\address{
Thomas~Reichelt\\
Mathematisches Institut \\
Universit\"at Heidelberg\\
Im Neuenheimer Feld 205\\
69120 Heidelberg\\
Germany}
\email{treichelt@mathi.uni-heidelberg.de}

\author{Uli Walther}
\address{ Uli~Walther\\
  Purdue University\\
  Dept.\ of Mathematics\\
  150 N.\ University St.\\
  West Lafayette, IN 47907\\ USA}
\email{walther@math.purdue.edu}

\author{Wenliang Zhang}
\address{ Wenliang~Zhang\\
  Department of Mathematics, Statistics, and Computer Science,
  University of Illinois at Chicago,
  Chicago, IL 60607}
\email{wlzhang@uic.edu}

\thanks{TR was supported by  DFG Emmy-Noether-Fellowship RE
  3567/1-2. UW was supported by the NSF and by Simons Foundation
Collaboration Grant for Mathematicians \#580839. WZ was supported by the NSF grant DMS\#1752081.}

\begin{abstract}
  We discuss for an affine variety $Y$ embedded in affine space $X$
  two sets of integers attached to $Y\subseteq X$ via local and de
  Rham cohomology spectral sequences. We give topological
  interpretations, study them in small dimension, and investigate to
  what extent one can attach them to projective varieties.
\end{abstract}
\maketitle

\setcounter{tocdepth}{3}
\tableofcontents

\section{Introduction}

\begin{ntn}
Throughout we will use the following conventions: $\KK$ will be a
field of characteristic zero,
\[
I\subseteq R_n=\KK[x_1,\ldots,x_n], \qquad X=\Spec(R_n)
\]
an ideal in the polynomial ring in $n$ indeterminates and the associated
affine space. Our default affine variety will
be
\[
Y:=\Var(I)\subseteq X, \text{ with complement } U=X\minus Y,
\]
and if $I$ is
homogeneous then
\[
\tilde Y:=\Proj(R_n/I)\subseteq \PP^{n-1}_\KK
\]
will be the projective scheme to $I$, with complement $\tilde
U:=\PP U =\PP^{n-1}_\KK\minus \tilde Y$. The homogeneous irrelevant ideal of
$R_n$ will
be denoted $\frakm=\ideal{x_1,\ldots,x_n}$ and $d$ will stand for
$\dim(Y)$.\schluss
\end{ntn}

Hartshorne's seminal work  \cite{Hartshorne-DRCAV} begins with
\begin{verse}
  The idea of using differential forms and their integrals to define
  numerical invariants of algebraic varieties goes back to Picard and
  Lefschetz\ldots
\end{verse}
and then outlines the development of this branch of mathematics until
the writing of his article on algebraic de Rham cohomology.
While originally the base field was the complex numbers $\CC$,
Hartshorne works in greater generality over fields $\KK$ of characteristic
zero. It has become clear since, particularly through the work of
Lyubeznik \cite{L-Dmods}, that Kashiwara's framework of $D$-modules is the
right set-up for these investigations. This article is a contribution
to this general theme, with the two main characters defined as follows.


\bigskip

If a variety $Y'$ can be embedded into a smooth $\KK$-variety $X'$ of
$\KK$-dimension $n$, one can
define the de Rham homology and cohomology functors of $Y'$ as 
\[
H_q^\dR(Y'):=\HH^{2n-q}_{Y'}(X',\Omega^\bullet_{X'}), \qquad\qquad
H^q_\dR(Y'):=\HH^q(X',\hatOmega^\bullet_{X'}).
\]
Here, $\HH(-)$ denotes hypercohomology functor on complexes of sheaves,
$\Omega^\bullet_{X'}$ is  the de Rham complex (relative to $\KK$) of $X'$,
and the hat denotes completion along $Y'$. Hartshorne proves that these
quantities do not depend on $X'$ or on the chosen embedding of $Y'$, and
demonstrates many interesting facts about these two functors.

We focus on de Rham homology for a moment, under the assumption that
$X'$ is affine. Then hypercohomology collapses to global sections since
the modules in $\Omega^\bullet_{X'}$ are coherent, equal to exterior
powers of the free $\calO_{X'}$-module $\Omega^1_{X'}$ of rank $n$ given by
the K\"ahler differentials on $X'$.  The set-theoretic
sections-with-support functor on a coherent sheaf agrees with
algebraic local cohomology.  In particular,
$\HH^{2n-q}_{Y'}(X',\Omega^i_{X'})$ is just local cohomology
$H^{2n-q}_{Y'}(\Omega^i_{X'})$ of the module of $i$-forms (identifying
sheaves with their global sections).

The sheaf $\omega_X:=\Omega^n_X$ has a natural right module structure
over the ring $\calD_X$ of $\KK$-linear differential operators on $X$.
The global sections of the sheaf of differential operators $\calD_X$
on $X$
are the elements of the Weyl algebra
\[
D_n=R_n\ideal{\del_1,\ldots,\del_n}
\]
where $\del_i$ stands for the partial differentiation operator
$\frac{\del}{\del x_i}$.  On the other hand, $\Omega^i_X$ is the free
$\calO_X$-module of rank ${n\choose i}$ generated by the symbols $\de
x_I=\de x_{j_1}\wedge\cdots\wedge \de x_{j_i}$ with $I\subseteq
2^{[n]}$ and $|I|=i$, and the global sections of $\omega_X$ are the
elements of the right $D_n$-module $D_n/\del\cdot
D_n:=D_n/(\del_1,\ldots,\del_n)D_n$.

Let us write $\Omega^\bullet_{\calD,X}$ for the Koszul co-complex on
$\calD_X$ generated by left multiplication by the derivations
$\del_1,\ldots,\del_n$. This is a free resolution of right
$\calD_X$-modules for $\omega_X$ shifted right $n$ steps, and yields an
explicit form of the de Rham cohomology functors
\[
H^i_\dR(-):=\HH^{i-n}(X,\omega_X\otimes_{\calD_X}^L(-))=
H^{i}(\Omega^\bullet_{\calD,X}\otimes_{\calD_X}(-)) 
\]
from the category of left $\calD_X$-modules to the category of
$\KK$-vector spaces. Since the constituents of
$\Omega^\bullet_{\calD,X}$ are $\calD_X$-free and $X$ is
$\calD$-affine, for each left $\calD_X$-module $\calM$ with global
sections $M$ one has
\[
H^i_\dR(\calM)=H^{i-n}((D_n/\del\cdot
D_n)\otimes^L_{D_n}M).
\]
If $\calM$ is holonomic, these vector spaces
are $\KK$-finite since they are the cohomology of the $\calD$-module
theoretic direct image functor under the map to a point \cite{HTT}.

We return to de Rham homology $\HH^{2n-q}_Y(X,\Omega^\bullet_X)$ with
$X$ equal to affine $n$-space.  Since $\Omega^j_X$ is finite free over
$\calO_X$, there is a natural identification
of $H^i_Y(\Omega^j_X)$ with
$\Omega^j_X\otimes_{\calO_X}H^i_Y(\calO_X)$.  The complex
\[
\ldots\to \Omega^{j-1}_X\otimes_{\calO_X}H^i_Y(\calO_X)\to 
\Omega^{j}_X\otimes_{\calO_X}H^i_Y(\calO_X)\to 
\Omega^{j+1}_X\otimes_{\calO_X}H^i_Y(\calO_X)\to \ldots
\]
with differential induced by the usual exterior derivative 
is quasi-isomorphic to the complex $\Omega^\bullet_{\calD,X}\otimes_{\calD_X}
H^i_Y(\calO_X)$. 

Since $X$ is affine, $\Gamma(X,-)$ induces a spectral sequence for
hypercohomology,
\begin{gather}\label{eqn-dR-ss}
H^p_\dR(H^q_Y(R_n))\Longrightarrow
\HH^{p+q}_Y(\Omega^\bullet_X)=H^\dR_{2n-p-q}(Y)
\end{gather}
that has been considered in \cite[Lemma 2.16]{Switala-dR-complete} in
the complete local case, and in \cite{Bridgland} in the context we are
working in.  We note that over the complex numbers, the abutment is
naturally equal to the reduced singular cohomology of $U:=X\backslash
Y$, so there is a spectral sequence
\begin{gather}\label{eqn-Cech-deRham}
E^2_{p,q}=H^p_\dR(H^q_Y(R_n))\Longrightarrow \tilde H^{p+q-1}(U;\CC)
\end{gather}
to the reduced cohomology of $U$.
For $I=\frakm$, the abutment is
$H^n_\dR(H^n_I(R_n))=\CC[2n]$, the reduced cohomology of the
$(2n-1)$-sphere shifted by one.
For details see for example \cite[p.~67]{Hartshorne-DRCAV} and
\cite[Thm.~3.1]{LSW}.

The articles \cite{Switala-dR-complete,Bridgland}
proceed to show that the $E_r$-pages,
$r\geq 2$, of these spectral sequences are isomorphic for all
embeddings of $Y$. 
In consequence, 
the terms on pages $r\geq2$ of \eqref{eqn-dR-ss}
are numerical invariants of $Y$.

\begin{dfn}
  Let $Y=\Var(I)$ be an affine variety embedded in $X=\Spec(R_n)$
  defined by the ideal $I\subseteq R_n=\KK[x_1,\ldots,x_n]$ over the
  field $\KK$ of characteristic zero. For $r\geq 2$, the
  \emph{$(r,p,q)$-\v Cech--de Rham number} of $Y$ is the dimension
  \[
  \rho_{p,q}^r(Y):=\dim_\KK (E^{n-p,n-q}_r)
  \]
  of the corresponding entry in the spectral sequence
  \eqref{eqn-dR-ss}. If $r=2$ we denote
  $\rho_{p,q}^r(Y)=H^{n-p}_\dR(H^{n-q}_I(R_n))$ by just
  $\rho_{p,q}(Y)$.
  \schluss  
\end{dfn}
Switala defined these for ideals in the power series ring
\cite[Dfn.~2.23]{Switala-dR-complete}; they are well-defined by
\cite[Prop.~2.17]{Switala-dR-complete} and \cite[Thm.~1.1]{Bridgland}.
The dimensions $\rho_{p,q}^r$ are invariant under field extensions,
and one can compute them algorithmically over any field of definition
for $I$, see \cite{OT1,OT2,Walther-cdrc}.


\bigskip

A related construction appeared  in \cite{L-Dmods}, where 
Lyubeznik shows that 
the socle dimensions of the $E_2$-terms of the Grothendieck spectral
sequence
\begin{gather}\label{eqn-lc-ss}
  E_2^{p,q}=  H^p_\frakm(H^q_I(R_n))\Longrightarrow H^{p+q}_\frakm(R_n)
\end{gather}
are independent of the closed embedding of $Y=\Spec(R_n/I)$ into any
affine space $\AA^n_\KK=\Spec(R_n)$ and uses it to define numerical
invariants
\[
\lambda_{p,q}(R_n/I):=\dim_\CC \Hom(R/\frakm,H^p_\frakm(H^{n-q}_I(R))).
\]
These numbers, known as \emph{Lyubeznik numbers}
have been investigated for nearly three decades and are indeed functions
of the ring $R/I$ (and do not depend on the presentation of $R/I$ as a
quotient of a polynomial ring). For detailed information on the history
and the \emph{status quo} we refer to the survey articles
\cite{BWZ-survey,WaltherZhang-survey}.

In this article we develop further the theory of the Lyubeznik numbers on
one side, and on the other describe a number of properties that the invariants
introduced by Switala and Bridgland enjoy.

More precisely, in the next section we study the \v Cech--de Rham
numbers for small dimension of $Y$, and investigate the collapse of
the corresponding spectral sequence. We identify some classes of
examples where this collapse happens on the $E_2$-page, and explain
why this is so for subspace arrangements, by stringing together
known results of Goresky--MacPherson, and \`Alvarez--Garc\'ia--Zarzuela. We
also explore the behavior of the \v Cech--de Rham numbers under
Veronese maps.

In the third section we discuss Lyubeznik numbers. We elaborate on the
results from \cite{RSW} by establishing some classes of projective
varieties $\tilde Y$ with Picard number one that have most Lyubeznik
numbers of the affine cone $Y$ independent of the embedding. That
includes determinantal varieties, certain toric varieties, and
horospherical varieties.  We also prove for certain projective
varieties of dimension four or less that their Lyubeznik numbers are
independent of the embedding.

\quad{\bf Some known facts.}
  
Since we will have to refer to them a few times, we state here some
results from the literature.

\begin{rmk}\label{rmk-L-argument}
  \begin{asparaenum}
    \item If $\KK$ is of characteristic zero, then local cohomology,
      algebraic de Rham cohomology, injective dimension, dimension,
      socle dimension all behave well under field extensions. Since
      all varieties are defined by a finite number of data, one can
      restrict all questions we discuss from the given field $\KK$ to
      a field of definition for $I$, and then extend to $\CC$. In
      particular, we can assume that $\KK=\CC$ whenever it is convenient.
  \item (\cite[(4.4.iii)]{L-Dmods})\quad Suppose $Y\subseteq
    X=\AA^n_\KK$ is an affine variety.  Then the local cohomology
    module $H^i_I(R_n)$ has support dimension at most $n-i$, and it
    vanishes if $i<c:=\codim(Y,X)$.  If $Y$ is equi-dimensional and
    $i>c$, then $H^i_I(R_n)$ has support dimension less than $n-i$.
%

  \item (\cite[Thm.~2.4]{L-Dmods})\quad If $M$ is a holonomic
    $D$-module, then $H^i_\frakm(M)$ is a finite sum of copies of the
    (Artinian, indecomposable) injective hull $H^n_\frakm(R_n)$ of
    $\KK=R_n/\frakm$. More generally, one has for all holonomic
    modules that 
    \[
    \injdim_R(M)\le \dim\Supp(M).
    \]
    Thus, all right derived functors of $R$-modules with derivation
    level greater than $n-i$ 
    vanish on $H^i_I(R)$,
    and those of derivation level
    $n-i$ vanish  if $I$ is equi-dimensional and $i>c$.
  \item We will also have to refer to equivariant $\calD_X$-modules. 
    For details on
    equivariance of $\calD$-modules, see for example
    \cite{LoerinczWalther}.
  \item Let $I\subseteq R=\KK[x_1,
  \dots,x_n]$ be a homogeneous ideal such that $\dim(R/I)\geq 2$. Assume that $\KK$ is separably closed. Hartshorne proved in \cite[Theorem 7.5]{Hartshorne-CDAV} that if $\Proj(R/I)$ is connected then $H^n_I(R)=H^{n-1}_I(R)=0$, and named this result the Second Vanishing Theorem. This theorem subsequently has been extended to the local settings as follows:
    Let $R$ be either a complete regular local ring of dimension $n$ that
    contains a separably closed coefficient field or an unramified complete regular local ring of dimension $n$ in mixed characteristic with a separably closed residue field. Let $I\subseteq
    R$ be an ideal. Then $H^n_I(R)=H^{n-1}_I(R)=0$ if and only if
    $\dim(R/I)\geq 2$ and the punctured spectrum of $R/I$ is connected,
    \cite{Ogus-LCDAV,PeskineSzpiro,HunekeLyubeznik, ZhangVanishingLCMixedChar}.
    \schluss
  \end{asparaenum}
\end{rmk}
  
The following is a special case of a more general result comparing
direct image to a point and restriction to a point.
\begin{lem}[{\cite[Lemma 3.3]{RW-weight}}]\label{lem-rw}
  Suppose $\KK=\CC$ and $X=\CC^n$. Assume that 
  $\calM$ is a regular holonomic $\calD_X$-module and its global sections form a
  standard graded $R_n$-module. Suppose further that $\calM$ is
  (strongly) equivariant as a $\calD$-module with respect to the
  $\CC^*$-action corresponding to this grading. Then its de Rham
  cohomology groups agree with the restriction groups to the origin of
  the holonomically dual module. In particular, the dimensions of
  these groups satisfy
  \[
  \dim_\CC (H^i_\dR(\calM))= \dim_\CC(\Hom_{R_n}(R_n/\frakm,
  H^{n-i}_\frakm(\DD(\calM))),
  \]
  where $\DD$ is the holonomic duality functor.
  \qed
\end{lem}

\section{\v Cech--de Rham numbers}

\subsection{Basic structure results}

Basic properties of the de Rham functor imply that
$\rho_{p,q}^r$ is zero for $p$ outside the interval $[0,n]$. On the
other hand, local cohomology $H^j_I(R)$ is nonzero only when
$\codim(I)\le j\le n$, and so $\rho_{p,q}^r$ is zero
for $q$ outside the interval $[0,\dim(Y)]$. Our first statement on these
numbers is that they are confined to a triangular region:

\begin{prop}\label{prop-rho-diag}
  The \v Cech--de Rham numbers satisfy for all $r\geq 2$ that
  \[
  \rho_{p,q}^r(Y) = 0 \qquad \text{ if } p>q.
  \]
\end{prop}

Before entering the proof we  set up some notation and collect several facts
and from \cite{Dimca-sheaves,HTT,KS} on constructible
sheaves and the Riemann--Hilbert correspondence. All spaces mentioned
in the sequel are assumed to be algebraic. 

\begin{rmk}\label{rmk-RH-stuff}
  Let $X$ be a smooth algebraic variety.
  \begin{asparaenum}

\item For any algebraic map $f$ between algebraic sets we denote, on
  the level of constructible sheaves, the usual direct and inverse
  image functors by $f_*$ and $f^{-1}$, and the proper direct and
  exceptional inverse image functors by $f_!$ and $f^!$
  respectively. For the sake of notational brevity, we mean by these
  symbols always the derived functors on the appropriate derived
  categories (so that,
  for example, we write $j_*$ instead of $Rj_*$ as a functor on the
  bounded derived category of constructible sheaves). This abuse of
  notation is common in the relevant literature.

\item On the level of $\calD$-modules, we will use $f_+$ and $f_!$
  for the usual and proper direct image functors, and $f^+$ and
  $f^\dagger$ for the usual and exceptional inverse image
  functors. For reference and comparison,
  our $\calD$-functors $f_+,f_!,f^+,f^\dagger$ are (in
  this sequence) denoted by $\int_f,\int_{f!},f^\dagger,f^\star$ in \cite{HTT}.

  \item The Riemann--Hilbert correspondence sets up an equivalence
  between the derived category of bounded complexes of
  $\calD_X$-modules with holonomic cohomology, and the derived
  category of bounded complexes
  of constructible sheaves $D^b_\cs(X)$. The correspondence is induced
  by the de Rham functor
  $\Omega^\bullet_{X^\an}\otimes_{\calD_{X^{\an}}}^L(-)$ computed on the
  analytic space attached to $X$.

  Under this correspondence,
  taking cohomology of a complex of $\calD_X$-modules corresponds to an
  operation on complexes of constructible sheaves that is denoted
  ${}^p\calH$; it is not the same as taking cohomology of complexes of
  constructible sheaves. We call \emph{perverse exact} any functor on
  the derived category that commutes with ${}^p\calH$.

  \item Suppose $f\colon X\to X'$ is a morphism of smooth algebraic
    varieties. Under the Riemann--Hilbert correspondence, the functors
    for $\calD$-modules correspond to those on constructible sheaves
    as follows:
    \[
    \DR_{X'}\circ f_+\simeq f_*\circ \DR_X;\qquad
    \DR_{X'}\circ f_!\simeq f_!\circ \DR_X;\qquad
    \DR_X\circ f^+\simeq f^!\circ \DR_{X'};\qquad 
    \DR_X\circ f^\dagger\simeq f^{-1}\circ \DR_{X'}. \]
    (The last two identifications are not misprints; for inverse
    images, the Riemann--Hilbert correspondence via the de Rham
    functor aligns 
    a regular inverse image with an exceptional one).
\item 
  Consider an open embedding $j\colon U\into X$ and a closed embedding
  $i\colon Z\into X$ where $Z$ is closed (and, \emph{a fortiori}, constructible)
  and where $U$ is the complement of $Z$
  in $X$. We have the following properties of induced
  functors for complexes of constructible sheaves:
\begin{itemize}
\item $i_!$ is exact and perverse exact;
\item $i^{-1}$ is exact but usually not perverse exact;
\item $i^!$ and $j_*$ are usually neither exact nor perverse exact;
\item $j^{-1}$ is exact and perverse exact;
\item $j_!$ is exact but usually not perverse exact.
\end{itemize}

\item In the situation of the previous item, we have the following
  distinguished triangles, Verdier dual to one another, in
  $D^b_\cs(X)$:
\begin{align}
 i_! i^! F^\bullet \to &F^\bullet \to j_* j^{-1} \overset{+1}{\to}\, , \notag \\
 j_! j^{-1} F^\bullet \to &F^\bullet \to i_! i^{-1} F^\bullet \overset{+1}{\to}\, . \notag
\end{align}

\item We will always denote by $a_S$ the map from a space $S$ to a
  point, which we occasionally denote with $\pt$ and occasionally
  identify with the vertex of a cone if a cone is present.\schluss
  \end{asparaenum}
\end{rmk}

We now enter the 
\begin{proof}[Proof of Proposition \ref{prop-rho-diag}]
It suffices to consider $r=2$.  We will use the Riemann--Hilbert
correspondence to translate $\rho_{p,q} = \dim_\CC(
H^{n-p}_{\dR}(H^{n-q}_I(R_n)))$ into the language of constructible
sheaves. The de Rham functor takes the local cohomology $H^{n-q}_I(\calO_X)$
to ${^p}\calH^{n-q} h_!h^!\CC_X[n] \simeq h_!({^p}\calH^{-q} \omega_Y)$
where $h:Y \rightarrow X=\AA^n_\CC$ is the canonical embedding,
$\CC_X[n]$ is the constant sheaf on $X$ shifted to the left by $n$ and
$\omega_Y = \DD \CC_Y$ is the (topological) dualizing complex
$R\Hom_{\cs}(\CC_Y,\CC_Y)$ for constructible sheaves on $Y$. (We use
$\DD$ also to denote Verdier duality, the operation corresponding to
holonomic duality under the Riemann--Hilbert correspondence).

By Remark \ref{rmk-L-argument}, 
$\dim(\Supp(h_!({^p}\calH^{-q} \omega_Y))) =
\dim(\Supp(H^{n-q}_I(R)))\le q$. Set $Y_q := \Supp({^p}\calH^{-q}
\omega_Y)$ and denote by
\[
i_q\colon Y_q \rightarrow Y \qquad\text{and}\qquad j_q\colon Y
\minus Y_q \rightarrow Y\] the embeddings of $Y_q$ and its
complement into $Y$ and denote by $\tilde{i}_q:Y_q \to X$ resp. $\tilde{j}_q: X \minus Y_q \to X$ the corresponding embeddings of $Y_q$ and its complement into $X$.

On the level of $\calD_X$-modules with support in $Y$, this gives an exact
triangle
\[
R\Gamma_{Y_q} \to \id \to (\tilde{j}_q)_+ (\tilde{j}_q)^\dagger\stackrel{+1}{\to}
\]
that corresponds via Riemann--Hilbert to
\[
(\tilde{i}_q)_! (\tilde{i}_q)^! \to \id\to (\tilde{j}_q)_* (\tilde{j}_q)^{-1}\stackrel{+1}{\to}
\]
for constructible sheaves.  

Since the support of ${^p}\calH^{-q} \omega_Y$ is $Y_q$, $(j_q)_*
j_q^{-1} ({^p}\calH^{-q} \omega_Y)=0$ and so $(i_q)_! (i_q)^!
({^p}\calH^{-q} \omega_Y) \simeq {^p}\calH^{-q} \omega_Y$.  In
particular $\calG_q := (i_q)^! ({^p}\calH^{-q} \omega_Y)$ is a
perverse sheaf on $Y_q$ satisfying $(i_q)_! \calG_q ={^p}\calH^{-q}
\omega_Y$.  As $\tilde{i}_q$ is proper, we find that
$(\tilde{i}_q)_*(\calG_q)=(\tilde{i}_q)_!(\calG_q)={^p}\calH^{-q}
\omega_Y$ and it follows that the hypercohomology of $\calG_q$ agrees
with the hypercohomology of ${^p}\calH^{-q}\omega_Y$.

But the hypercohomology
$\HH^k(\calF)$ vanishes for $k \not \in [-q,0]$ and every perverse sheaf
$\calF$ on $Y_q$ (see e.g. \cite[Corollary 5.2.18]{Dimca-sheaves} and
\cite[Proposition 5.2.20]{Dimca-sheaves}. The claim follows now from
\[
\rho_{p,q}^2 = \dim( H^{n-p}_{\dR}(H^{n-q}_I(R_n))) = \dim ( \HH^{-p}({^p}\calH^{-q}\omega_Y)) = \dim(\HH^{-p}(\calG_q)).
\]
\end{proof}

\begin{rmk}
  If $I$ is homogeneous, a more elementary argument can be
  made. Indeed, the local
  cohomology module $H^p_I(R_n)$ is Eulerian (see Definition \ref{dfn-Eulerian}
  below). This means that the Lie algebra action induced from the
  differentiating the $\CC^*$-action corresponding to the grading  
  agrees with the action of the Lie algebra via the morphism from the
  universal enveloping algebra to the Weyl algebra. In other words,
  $H^p_I(R_n)$ is equivariant and \cite[Lem.~3.3]{RW-weight}
  applies in the form of Lemma \ref{lem-rw}.

  Holonomic duality does not affect the support of the underlying
  module and preserves the category of regular holonomic $\mathcal{D}$-modules.  Thus, 
  for graded $I$ we have that $H^{n-p}_\dR(H^{n-q}_I(R_n))=0$ if and only if
  $H^p_\frakm(\DD(H^{n-q}_I(R_n))=0$.  This latter vanishing holds
  whenever $\dim(\Supp(H^{n-q}_I(R_n)))< p$ since $H^p_\frakm(\calM)$
  is zero whenever $p>\dim(\Supp(\calM))$, \cite{Grothendieck}.  But
  by Remark \ref{rmk-L-argument}, $\dim(\Supp(H^{n-q}_I(R_n)))\le q$.\schluss
\end{rmk}

If one pictures the $\rho^r_{p,q}$ as a table, it thus takes the
following general form, assuming that $Y$ is embedded into
$\AA^n_\KK$, cut out by the ideal $I\subseteq R$ of dimension $d$:
\[
P^r(Y)=((\rho^r_{p,q})):=
\begin{pmatrix}\rho^r_{0,0}&\cdots&\cdots&\rho_{0,d}^r\\
  0&\ddots&&\vdots\\
  \vdots&\ddots&\ddots&\vdots\\
  0&\cdots&0&\rho^r_{d,d}\end{pmatrix}
\]
Here, $p$ is the row index counting downward, $q$ the column index
counting towards the right, 
and the arrows of the \v Cech--de Rham spectral sequence 
point North to Northeast.

\subsection{Degeneration}

Switala raised in \cite[Question 8.2]{Switala-dR-complete} the
following question for a complete local ring $A$ with coefficient
field $\KK$ of characteristic zero:
\begin{verse}
  ``Does the local Hodge--de Rham homology spectral sequence
    degenerate at $E_2$ ?''
\end{verse}
One can ask a similar question for the affine scenario. We discuss
interesting classes where this question has a positive answer.

\begin{exa}
  Suppose $Y=\Var(I)$ is a complex subspace arrangement. Let $\posP_Y$ be its
  \emph{intersection lattice}, the collection of all possible
  intersections of the components of $Y$, ordered by inclusion. (This
  differs from standard notation in arrangement theory, where the
  order is the reverse).
  We agree that $\posP_Y$ has a unique maximal element corresponding to the
  ambient space, but it may have several minimal elements as we do not
  insist that $I$ be homogeneous (so, the arrangement may not be central).

  It is well-known that the cohomology of the complement $\CC^n\minus
  Y$ is determined by the combinatorics of $\posP$: building on work
  of Brieskorn, Orlik and Solomon \cite{OrlikSolomon-Invent80}
  showed that the cohomology algebra of this complement is given by a
  purely combinatorial algebra constructed from the matroid of the
  arrangement.

  Goresky and MacPherson \cite[III, Thm.~1.3]{GoreskyMcPherson-book}
  proved that the Betti numbers of the complement can be computed as a
  sum of non-negative integers, one for each element of
  $\posP_Y$. Here, the integers for each flat $p\in\posP$ are computed
  as Betti numbers of the simplicial complex $K(>p)$. (While Goresky
  and MacPherson phrase this in terms of relative homology for the pair
  $(K(\geq p),K(>p))$, the space $K(\geq p)$ is contractible and one
  can convert into an absolute homology without harm).

  \`Alvarez, Garc\'ia and Zarzuela established the degeneration 
  on page two of a certain spectral sequence
  \begin{eqnarray}\label{eqn-MV-ss}
  E^{-i,j}_2={\ilim_{p\in \posP_Y}}^{(i)}H^j_{I_p}(R_n)&\Longrightarrow&
  H^{j-i}_I(R_n)
  \end{eqnarray}
  for the local cohomology groups $H^\bullet_I(R_n)$, the inverse limits
  being taken over the poset $\posP_Y$ viewed as a category with a
  morphism for each containment. In
  \cite[Thm.~1.2]{AlvarezGarciaZarzuela}, the structure of the derived
  inverse limits is explained as direct sums of modules $H^j_{I_p}(R)$
  with $\codim(I_p)=j$ and multiplicity given by the topological Betti
  numbers of $K(>p)$. In \cite[Cor.~1.3]{AlvarezGarciaZarzuela}, this
  is used to give a formula for the cohomology groups of the
  complement of $Y$, by translating the Goresky--McPherson formula.

  The affine complement of an affine space is homotopy equivalent to a
  sphere, hence applying the de Rham functor to a module of the form
  $H^j_{I_p}(R)$ gives exactly one (reduced) cohomology group. Thus,
  the entries of the $E_2$-page of the \v Cech--de Rham spectral
  sequence \eqref{eqn-Cech-deRham} correspond exactly to the
  composition factors of $H^{j-i}_I(R)$ in the spectral sequence
  \eqref{eqn-MV-ss} from \cite{AlvarezGarciaZarzuela} on one side, and
  to the direct summands for $H^\bullet(\CC^n\minus Y)$ in
  \cite{GoreskyMcPherson-book} on the other. It follows that for
  complex subspace arrangements $Y$ the \v Cech--de Rham spectral
  sequence collapses on the $E_2$-page.  \schluss
\end{exa}

In small
dimensions we show that Switala's question has a positive answer as
well. 

\begin{prop}\label{prop-cdR-dim3}
  If $I$ is homogeneous and $\dim(\Var(I))\le 3$ then the \v Cech--de
  Rham spectral sequence degenerates at $E_2$.
\end{prop}
\begin{proof}
  Let $Y$ be of dimension 2 or less. If follows from 
  Proposition \ref{prop-rho-diag} that no nonzero differential can exist in the spectral sequence.

  Let now $\dim(Y)=3$.
  Then Proposition \ref{prop-rho-diag} implies that then there might be
  at most one nonzero differential,
  \begin{gather}\label{eqn-d2map}
    d_2\colon H^{n-2}_\dR(H^{n-2}_I(R_n))
    \to H^n_\dR(H^{n-3}_I(R_n)),
    \end{gather}
  linking $\rho_{2,2}$ and $\rho_{0,3}$.
  
  Assume for the time being that $Y$ is purely 3-dimensional. Remark
  \ref{rmk-L-argument} says that $\dim\Supp(H^{n-i}_I(R_n))<i$ for
  $i<3$. In particular, by Lemma \ref{lem-rw},
  $\dim H^{n-2}_\dR(H^{n-2}_I(R_n))$ equals the socle dimension of
  $H^2_\frakm(\DD H^{n-2}_I(R_n))=0$. Thus, the degeneration of the spectral
  sequence is forced.

  Now relax the equi-dimensionality condition and let $Y_3$ and $Y'$ be
  the 3-dimensional and smaller dimensional components of $Y$
  respectively. Then $Y_3\cap Y'$ is of dimension 1 or less, and the
  Mayer--Vietoris sequence implies that $H^{n-3}_{Y_3}(R_n)\oplus
  H^{n-3}_{Y'}(R_n) =H^{n-3}_Y(R_n)$ and that there is a short exact
  sequence
  \[
  0\to H^{n-2}_{Y_3}(R_n)\oplus H^{n-2}_{Y'}(R_n)\to H^{n-2}_Y(R_n)\to C\to
  0
  \]
  where $C$ is a (graded) submodule of $H^{n-1}_{Y_3\cap Y'}(R_n)$. In
  particular, the dimension of the support of $C$ is one or less
  by Remark \ref{rmk-L-argument}
  and so $H^{\le n-2}_\dR(C)$ is zero, being dual to
  the socle of $H^{\geq 2}_\frakm(C)=0$.

  Applying the de Rham functor, the resulting long exact sequence
  shows that $H^{n-2}_\dR(H^{n-2}_{Y_3}(R_n)\oplus H^{n-2}_{Y'}(R_n))$
  equals $H^{n-2}_\dR(H^{n-2}_Y(R_n))$.
  Then the map \eqref{eqn-d2map} is the direct sum of the
  corresponding $d_2$-morphism for
  $Y_3$ and for $Y'$ separately. But it is zero 
  on $H^{n-2}_\dR(H^{n-2}_{Y_3}(R_n))$ since the source of $d_2$ is zero
  in that case, and it is zero on $H^{n-2}_\dR(H^{n-2}_{Y'}(R_n))$ since
  the target is zero in that case.
\end{proof}

\subsection{Affine complements}

In the next two subsections we investigate to what extent the
cohomology of the affine complement, or its table of \v Cech--de Rham
numbers, of a homogeneous variety $Y$ is determined by the associated
projective variety $\tilde Y$. We start with looking at the top
cohomology group of the affine complement, and then investigate the
affine complement under Veronese maps. In the process we review some
algorithmic ideas that lead to a condition on the de Rham classes of
graded $\calD$-modules on affine space.

So, throughout, $\tilde Y$ is a projective variety and $Y\subseteq
\CC^n$ is a cone for $\tilde Y$.

\subsubsection{High cohomology groups of the affine complement}

\begin{rmk}\label{rmk-Ogus-1}
  Let $\tilde Y$ be a projective variety with cone $Y=\Spec(R_n/I)$.
  The following facts are due to Ogus
  \cite{Ogus-LCDAV} Let
\[
f_Y:=\min(k\in\NN|H^\ell_I(R)\text{ is Artinian for all }\ell>k)
\]
and
\[
v_Y:=\min(k\in\NN|H^\ell_I(R)\text{ is zero for all }\ell>k).
\]
\begin{enumerate}
  \item The number $n-f_Y$ is intrinsic to $\tilde Y$, it does not
    depend on the choice of the cone $Y$,\cite[Thm.~4.1]{Ogus-LCDAV}.
  \item The number $n-v_Y$ is intrinsic to $\tilde Y$, it does not
    depend on the choice of the cone $Y$,\cite[Thm.~4.4]{Ogus-LCDAV}
    and the remark following it.
\end{enumerate}
In particular,
\[
\rho^r_{p,q}=0 \quad \left\{\begin{array}{ll}\text{ if }&q<n-\nu_Y,\\\text{
  or }&p>0\text{ and }q<n-f_Y.\end{array}\right.
\]
\schluss
\end{rmk}  

We show next that in fact the top de Rham cohomology group of the
affine cone complement is usually determined by $\tilde Y$.

\begin{lem}\label{lem-topLocalDeRham}
Let $X=\AA^n_\CC$ and suppose $\tilde Y\subseteq \PP X=\PP^{n-1}_\CC$
is defined by the homogeneous ideal $I\subseteq
R_n:=\Gamma(X,\calO_X)$.  Let $Y=\Var(I)\subseteq X$ and assume that
$Y$ has codimension at least two. Then the index and the dimension of
the top non-vanishing de Rham cohomology group of $U:=X\minus Y$ is
encoded on $\tilde Y$.
\end{lem}

\begin{proof}
    We recall Alexander duality, compare \cite[V.6.6]{Iversen}: if $\PP$
    is a $\CC$-orientable manifold and $\tilde Y$ a closed subset then
    the topological local cohomology group $H^i_{\tilde Y}(\PP;\CC)$ is
    canonically identified with the $\CC$-dual of the cohomology with
    compact support $H^{2\dim_\CC \PP-i}_c(\tilde Y;\CC)$. If $\tilde Y$ is,
    in addition, compact, the latter is just $H^{2\dim_\CC \PP-i}(\tilde
    Y;\CC)$.

  On the other hand, \cite[II.9.2]{Iversen} states the existence of a
  long exact sequence
  \begin{gather}\label{eqn-les-Iversen}
  H^i_{\tilde Y}(P;\calF)\to H^i(P;\calF)\to H^i(P\minus
  \tilde Y;\calF)\stackrel{+1}{\to}
  \end{gather}
  where $\calF$ is a sheaf of Abelian groups on $P$ and we ease
  notation by ignoring the pull-backs of $\calF$ to $\tilde Y$ and its
  complement
  respectively. Notice that one can get this long exact sequence by applying hypercohomology to the first triangle in Remark \ref{rmk-RH-stuff} (4).  We use these with  $\PP=\PP X$ and $\tilde Y$ as above.

Via Poincar\'e duality, the map $H^{i}_{\tilde Y}(\PP X;\CC)\to
H^{i}(\PP X;\CC)$ becomes $H^{2n-2-i}(\tilde Y;\CC)^\vee\to
H^{2n-2-i}(\PP X;\CC)^\vee$. This is the dual of $H^{2n-2-i}(\PP
X;\CC)\to H^{2n-2-i}(\tilde Y;\CC)$ induced by restriction from $\PP
X$ to $\tilde Y$. The restriction $H^{i}(\PP X;\CC)\to H^{i}(\tilde
Y;\CC)$ is injective
\footnote{The cohomology fundamental class of $Y$
  in $H^{2 \dim_\CC(\tilde Y)}(\PP X;\CC)$ evaluates on the homology
  class of a generic $\PP^{n-1-\dim_\CC(\tilde Y)}\subseteq \PP X$ to
  the 0-cycle given by the intersection of $\tilde Y$ with that
  generic subspace. But this intersection is the degree of $\tilde Y$,
  hence positive. Thus the restriction of the class represented by
  this subspace on $\PP X$, a generator of $H^{2\dim_\CC(\tilde
    Y)}(\PP X;\CC)$, to $\tilde Y$ is nonzero. But cohomology of
  projective space is a polynomial algebra in the hyperplane section,
  and if the $\dim_\CC(\tilde Y)$-power of the hyperplane restricts to
  a nonzero class on $\tilde Y$ then so do all smaller powers.}
for
$i\le 2\dim \tilde Y$ and necessarily zero for $i>2\dim_\CC(\tilde Y)$
since $\tilde Y$ is a CW-complex of dimension $2\dim_\CC(\tilde
Y)$. Thus, one can determine from the topological Betti numbers of
$\tilde Y$ alone the sizes of the kernels of the left-most morphisms
in display \eqref{eqn-les-Iversen}. This in turn determines the sizes of the
cohomology groups of $\PP U:=\PP X\minus \tilde Y$.

%
%

   As $\codim(Y)\geq 2$,
  $U$ is simply connected by \cite[Thm.~2.3]{Godbillon}.
  Thus, the $\CC^*$-fiber bundle $U\to \PP U$
  has a Leray spectral sequence
  \[
  H^i(\PP U;H^j(\CC^*;\CC))\Longrightarrow H^{i+j}(U;\CC)
  \]
  in which the coefficients on the left are global (in a trivial
  vector bundle). Let $m$ be the largest index with $H^m(\PP
  U;\CC)\neq 0$. Since all differentials out of and into $H^{m}(\PP
  U;H^1(\CC^*;\CC))\neq 0$ are zero, $m+1$ must be the largest index with
  $H^{m+1}(U;\CC)\neq 0$ and $\dim_\CC
  H^{m}(\PP U;\CC)=\dim_\CC H^{m+1}(U;\CC)$.
\end{proof}

\begin{cor}\label{cor-topLocalDeRham}
  Let $Y$ be an affine variety defined by the homogeneous ideal
  $I\subseteq R_n=\CC[x_1,\ldots,x_n]$. If the local cohomology group
  $H^\ell_I(R)$ is Artinian and $H^{>\ell}_I(R_n)=0$ then the socle
  dimension $s$ of $H^\ell_I(R)$ is a function of the projective variety
  $\tilde Y=\PP(Y)$ and does not depend on the choice of the cone $Y$.
\end{cor}
\begin{proof}
  Let $X=\Spec(R)$ and set $U=X\minus Y$.
  By \cite[Thm.~3.1]{LSW}, $s=\dim_\CC H^{n+\ell-1}(U;\CC)$, and $U$
  has no higher non-vanishing cohomology groups. Now use the previous lemma.
\end{proof}

\begin{rmk}
  Even if $H^\ell_I(R)$ is not  Artinian, the dimension $\dim_\CC
  H^n_\dR(H^\ell_I(R_n))$ is encoded by $\tilde Y$.\schluss
\end{rmk}

\subsubsection{Integrals of Eulerian modules}


 We investigate  next 
 to what extent the $\rho_{p,q}^r$, or the abutment terms
 $H^\dR_\bullet(Y)$ of the \v Cech--de Rham spectral sequence are
 independent of the cone $Y$ (\emph{i.e.}, the line bundle $\calL$ on
 $\tilde Y$ that induces the cone). In the following we show that
 replacing $\calL$ by a power of itself does not change the
 $H^\dR_\bullet(Y)$.

For this we
give an account on the main results on algorithmic
computation of the integral of a $D_n$-module along
$\del_1,\ldots,\del_n$. See \cite{OT1,OT2,Walther-cdrc} for
details, and a generalization to the case when $M$ is a bounded
complex of finitely generated modules that has holonomic cohomology.

We define a grading $\gr^i_\tildeV(D_n):=\{P\in D_n|\deg(P)=i\}$ on
$D_n$ by setting
\[
\deg(x_j)=1=-\deg(\del_j)
\]
for all $1\le j\le n$. With it we define a filtration on $D_n$ by
\[
\tildeV^k(D_n)=\sum_{i\le k}\gr^i_{\tildeV}(D_n).
\]

Let $M$ be a $D_n$-module, finitely generated by elements
$m_1,\ldots,m_r$, and choose integers $s_1,\ldots,s_r$. Then define a
filtration on $M$ by setting
\[
\tildeV^k(M)=\sum_{i=1}^r \tildeV^{k-s_i}(D_n)\cdot m_i.
\]
Denote the operator $-\sum_{j=1}^n \del_j\cdot x_j$ by $\tilde E$.

 It is a result of Kashiwara \cite{Kashiwara-holonomicII} that when $M$ is holonomic there is a
\emph{$b$-function for integration} $\tilde b_M(s)$. This is a
univariate polynomial that satisfies
\begin{gather}\label{eqn-bfu}
\tilde b_M(\tilde E+n+k)\cdot \tildeV^k(M)\subseteq \tildeV^{k-1}(M)
\end{gather}
for all $k\in\ZZ$. 
We describe now
ideas that lead to a proof for Proposition
\ref{prop-integrate-Eulerian} below.

As before, let $\omega_n$ be the right $D_n$-module $(D_n/\del\cdot
D_n)$ where $\del=\{\del_1,\ldots,\del_n\}$.
This is a free rank one $R_n$-module, and can be naturally
identified with the $D_n$-module $\Ext^n_{D_n}(R_n,D_n)$,
and with the global
sections of the right $\calD_X$-module  of
top differential forms $\calO_X\cdot \de x_1\wedge\ldots\wedge \de x_n$, \cite{HTT}.
Give it a
$\tildeV$-filtration by placing the generator $1+\del\cdot D_n$ into
$\tildeV$-level $n$.

The $D$-module theoretic direct image functor $\pi_+$ for
the projection map $\pi\colon \CC^n\to\CC^0$
can on global sections be identified with $\omega_n\otimes_{D_n}^L (-)$
shifted by $n$, computing the
Tor-functors against $\omega_n$.  This derived tensor product can be
viewed as the tensor product of $\omega_n$ with a free $D_n$-resolution
$F^\bullet$ of the input module $M$, or of a free resolution
$K^\bullet$ of $\omega$ with $M$, or of the tensor product of
$K^\bullet$ with $F^\bullet$.  There are natural morphisms from the
last scenario to the two former ones that induce isomorphisms on
cohomology.

One major difficulty in identifying $\pi_+(M)$ is that its homology
consists of
finite-dimensional vector spaces with no further module structure,
while the modules that appear in the complex  are infinite-dimensional
vectors spaces with no further module structure.

A free resolution $F^\bullet$ of $M$ is \emph{$\tildeV$-strict} if
each $F^i$ is equipped with a $\tildeV$-filtration $\tildeV (F^i)$
such that every differential $\delta^i\colon F^i\to F^{i+1}$ satisfies
$\delta^i(\tildeV^k(F^i))\subseteq \tildeV^k(F^{k+1})$ and moreover
$\delta^i(F^i)\cap \tildeV^k(F^{i+1})=\delta^i(\tildeV^k(F^i))$. It is
a theorem of algorithmic algebraic analysis that finitely generated
$V$-filtered $D_n$-modules do allow $\tildeV$-strict resolutions of
finite length.
The $\tildeV$-filtration on $F^\bullet$ induces a quotient filtration
on $\omega\otimes_{D_n}F^\bullet$. This filtered complex may not be
strict anymore, but still the morphisms will respect the
filtration. The $\tildeV$-filtration on $\omega_n\otimes_{D_n}F^\bullet$ is
bounded below while on $F^\bullet$ it is not. Moreover,
$\gr^k_\tildeV(F^i)$ is infinite dimensional over $\CC$, while each
$\gr^k_\tildeV(\omega_n\otimes_{D_n}F^i)$ is $\CC$-finite.
Nonetheless, the
$\CC$-dimension of each $\omega_n\otimes_{D_n}F^i$ is still infinite.

Let $\ell$ be the largest and $s$ the
smallest integral root of the $b$-function $\tilde b_M(s)$.

\begin{thm}[Integration Theorem \cite{OT1,OT2}]\label{thm-int-thm}
  With notation as introduced above, the morphisms
  \[
  \omega_n\otimes_{D_n} F^\bullet\hookleftarrow
  \tildeV^\ell(\omega_n\otimes_{D_n} F^\bullet)\onto
  \tildeV^\ell(\omega_n\otimes_{D_n} F^\bullet)/
  \tildeV^{s-1}(\omega_n\otimes_{D_n} F^\bullet)
\]
are quasi-isomorphisms.

In other words, every cohomology class of
$\Tor_\bullet^{D_n}(\omega_n,M)$ has a representative inside
$\tildeV^\ell(\omega_n\otimes_{D_n} F^\bullet)$, and the complex
$\tildeV^{s-1}(\omega_n\otimes_{D_n} F^\bullet)$ is exact.
\end{thm}

Note that the subquotient complex $\tildeV^\ell(\omega_n\otimes_{D_n}
F^\bullet)/ \tildeV^{s-1}(\omega_n\otimes_{D_n} F^\bullet)$ is, in
contrast to $\omega_n\otimes F^\bullet$, 
$\CC$-finite,
reducing the
computation of $\pi_+(M)$ to finite-dimensional linear algebra in this
subquotient complex.

One can now just as well resolve $\omega_n$ and $M$, or just $\omega_n$,
and obtain other complexes that represent $\omega_n\otimes ^L_{D_n}M$. A
natural resolution for $\omega_n$ is the cohomological Koszul complex
$K^\bullet$ on the left-multiplications on $D_n$ by the various
$\del_j$. (So, $K^\bullet$ is the complex of global sections of
$\Omega^\bullet_{\calD,X}$). 
The module
$K^\ell$ has a natural generating set given by the size-$\ell$-subsets
of $1,\ldots,n$. We place these generators in $\tildeV$-level $\ell$
and extend $\tildeV$ to each $K^\ell$ by $D_n$-linearity. Since
$\del_i$ is in $\tildeV$-level $-1$, this produces a $\tildeV$-strict
resolution of $\omega_n$.  Having resolutions $K^\bullet, F^\bullet $
with $\tildeV$-filtration, there is an induced $\tildeV$-filtration on
$K^\bullet\otimes_{D_n}F^\bullet$.

The complex $K^\bullet\otimes_{D_n}M$ is sometimes called the (affine,
global) \emph{de
  Rham complex} of $M$. If $M$ is a space of functions on which one
can differentiate,  multiplication by $\del_i$ in $K^\bullet$
corresponds to differentiation by $x_i$ in the usual de Rham complex.

\medskip
Suppose now that $M=\bigoplus_{\ell\in\ZZ}M_\ell$ is a graded module
over the graded ring $D_n$, with homogeneous generators
$m_1,\ldots,m_r$ of degrees $s_1,\ldots,s_r$.  Choosing the degrees of the
generators as shifts (\emph{i.e.}, $s_i=\deg(m_i)$) for the
$\tildeV$-filtration on $M$ one obtains a direct sum of the graded components
of $M$,
\begin{gather}\label{eqn-gradedV}
\tildeV^k(M)=\sum_{i=1}^r \tildeV^{k-s_i}(D_n)\cdot m_i=\bigoplus_{\ell\le k} M_\ell.
\end{gather}
Since the twisted Euler operator $\tilde E=-\sum_{j=1}^n\del_jx_j$ is
$\tildeV$-homogeneous of degree zero, the defining equation \eqref{eqn-bfu}
becomes
\[
\tilde b_M(\tilde E+n+k)\cdot M_k =0
\]
for all $k\in\ZZ$.

For $\tildeV$-graded $M$ one can arrange the resolution $F^\bullet$ to
respect this grading, and $K^\bullet$ is graded in any case.  If now
$\eta_F$ is a cohomology class generator in
$H^i(\omega_n\otimes_{D_n}F^\bullet)$, one can lift it into
$K^n\otimes_{D_n}F^i$ and then chase it into a class $\eta_K$ of
$K^\bullet\otimes_{D_n}M$, since Tor is a balanced functor.  The
grading of the resolutions involved implies that the $\tildeV$-level
of this class in $K^\bullet\otimes_{D_n}M$ is the same as the
$\tildeV$-level of $\eta_F$ in $\omega_n\otimes_{D_n}F^\bullet$.

We recall the notion of an Eulerian $D_n$-module.
\begin{dfn}[\cite{MaZhang-Eulerian}]\label{dfn-Eulerian}
  The graded $D_n$-module $M=\bigoplus_{i\in\ZZ} M_i$ is
  \emph{Eulerian} if for every homogeneous $m\in M_i$ one has
  $(\sum_{j=1}^nx_j\del_j)m=i\cdot m$.

  In terms of $\tilde E$ this is equivalent to $(\tilde E
  +n+\deg(m))m=0$.
  \schluss
\end{dfn}
Eulerian $D_n$-modules are a very special case of Brylinski's
\emph{monodromic} modules, which are those on which the Euler operator
has a minimal polynomial. They include (iterated) local cohomology
modules
$H^{i_1}_{I_1}(\ldots(H^{i_k}_{I_k}(R_n)\ldots)$ for homogeneous ideals
$I_1,\ldots,I_k$. 
\begin{prop}\label{prop-integrate-Eulerian}
  Let $M$ be a finitely generated Eulerian $D_n$-module. Then every nonzero
  cohomology class of $\omega_n\otimes^L_{D_n}M$ has degree zero.
\end{prop}
\begin{proof}
  Since the module is Eulerian, we have $(\tilde E +n+\deg(m))m=0$ for
  every homogeneous $m\in M$.  We put the $\tildeV$-filtration on $M$
  that is induced by a finite set of homogeneous generators as in
  \eqref{eqn-gradedV}, with shifts $s_i=\deg(m_I)$. Then, a
  $b$-function for integration is given by $\tilde b(s)=s$.
  The conclusion is immediate from the Integration Theorem \ref{thm-int-thm}.
\end{proof}
\begin{rmk}
  Let us call \emph{quasi-Eulerian} a graded monodromic $D_n$-module
  $M$.  Then one can easily generalize Proposition
  \ref{prop-integrate-Eulerian} to: if $M$ is quasi-Eulerian then the
  degree of every cohomology class of $\omega_n\otimes^L_{D_n}M$ must be
  an integral root of the minimal polynomial of $\tilde E$ on $M$.

  There is a version of the Integration Theorem 
  for complexes of holonomic modules (more generally, for complexes
  that have a $b$-function for integration), see
  \cite{Walther-cdrc}. This allows a further generalization to finite
  graded complexes with quasi-Eulerian cohomology.  \schluss
\end{rmk}

We now consider the Eulerian $D_n$-module that arises as the
localization $M=R_n[1/f]$ of $R_n$ at a homogeneous polynomial $f$. It
is clear that this is an Eulerian module since the Euler operator $E$
acts on a rational homogeneous function of degree $k$ by
multiplication with $k$.
Thus, $K^\bullet\otimes_{D_n}M$ is $\tildeV$-graded and 
every class in $\omega_n\otimes ^L_{D_n}M$ has native degree zero.

If one reads elements of $K^\ell\otimes M$ as differential $\ell$-forms on
$M$, this implies that the cohomology of $K^\bullet\otimes_{D_n}M$ is
spanned as vector space by differential forms of degree zero: forms of
the type \[
\sum_{|I|=\ell\atop
  I\subseteq\{1,\ldots,n\}} \frac{g_I\de x_I}{f^{k_I}}
\]
where $\de
x_I=\wedge_{i\in I}\de x_i$, where $g_I$ is a homogeneous element of
$R_n$, and where $\deg(g_I)+\ell=k_I\cdot\deg(f)$.
Similarly, integrating a graded complex $M^\bullet$ with Eulerian cohomology
modules yields a de Rham complex of $M^\bullet$ with cohomology groups
concentrated in degree zero.

\begin{cor}
  If $I$ is a homogeneous ideal then the de Rham cohomology of the
  affine complement $U(I)=X\minus \Var(I)$ of the affine variety
  $\Var(I)\subseteq X:=\CC^n$ is generated by chains of differential
  forms of degree zero. Moreover, the de Rham cohomology groups
  $H^i_\dR(H^j_I(R))$ all are concentrated in degree zero. \qed
\end{cor}
\begin{proof}
  The Grothendieck comparison theorem asserts that the cohomology of
  $K^\bullet\otimes_{D_n}\check  C^\bullet$ is the de Rham cohomology of
  $U(I)$. The rest follows from Proposition \ref{prop-integrate-Eulerian}.
\end{proof}

\begin{rmk}
  Since multiplication by $\CC\ni\lambda\neq 0$ is an isomorphism on
  $U(f)$, the de Rham cohomology of $U(f)$ of a divisor is spanned
  by homogeneous differential forms (homothety eigenvectors) for all
  homogeneous $f\in R_n$. Alex
  Dimca pointed out that 
  path-connectedness of $\CC^*$ implies that this
  multiplication is in fact homotopy equivalent to the identity, and
  thus does not change the class. Hence, the cohomology of $U(f)$ must
  be eigenvectors to eigenvalue 1, and thus of degree zero.
\schluss
\end{rmk}

\subsection{On Veronese maps}
\label{subsec-veronese-complement}

Throughout this subsection, $2\le d,n\in\NN$.
Let
\[
v_n^d\colon X=\AA^n_\CC\to\AA^N_\CC=:W
\]
be the $d$-th Veronese morphism on the affine level, so
$N={n+d-1\choose n-1}$. If $n,d$ are understood, we abbreviate
$v_n^d$ to just $v$. We set $X':=v(X)\subseteq W$, $W^\circ=W\minus
\{0\}$, $X^\circ:=X\minus\{0\}$ and ${X'}^\circ:=X'\minus\{0\}$.

Let $R_n=\CC[x_1,\ldots,x_n]=\calO_X(X)$ and $R_N=\CC[\{y_S\mid S\in
  \NN^n, |S|=d\}]=\calO_W(W)$.  Let $I\subseteq R_n$ be a homogeneous
ideal and $Y$ the associated variety. Denote $U$ the complement
$X\minus Y$, and let $Y',U'$ the images of $Y,U$ under $v$. Let $U_W$
be the complement $W\minus v(Y)$.  We wish to compare here the cohomology
of the affine complements of $Y$ and $Y'$.

\bigskip

Note that $v^\#\colon R_N\to R_n$ sends
$y_S\mapsto x^S$ in multi-index notation. The $d$-th roots of unity
$\mmu_d$ act diagonally on $X$, as well as on every other variety of a
homogeneous ideal of $R_n$, by multiplication on each $x_i$. Moreover, $v$
is the orbit map to this action, followed by inclusion into $W$.  The
image of $v$ has a unique isolated singularity at the origin, and $v$
is a $d:1$ covering of ${X'}^\circ$ by $X^\circ$.

Note that $\mmu_d$ is the covering group of the map $U\to U'$, and its
order $d$ is nonzero in $\CC$. Under these circumstances,
$H^\bullet(U';\CC)$ is the group of $\mmu_d$-invariants in
$H^\bullet(U;\CC)$. Using the de Rham manifestation of $H^\bullet(U;\CC)$, in
which we showed that every class has a representative that is of
degree zero, the entire space $H^\bullet(U;\CC)$ is
$\mmu_d$-invariant, so that
\[
H^\bullet_\dR(U;\CC)=H^\bullet_\dR(U';\CC).
\]

In what follows, we replace de Rham cohomology by singular cohomology,
since we will have need to step outside the category of smooth algebraic
varieties. Known comparison theorems over $\CC$ assure functorial
isomorphisms between these cohomology theories whenever both exist.

It will turn out to be useful to know the cohomology of $W\minus
X'=W^\circ\minus {X'}^\circ$.  Note
that ${X'}^\circ$ has the homology of the homotopy $(2n-1)$-sphere
$X^\circ$ and is a closed submanifold of the $2N$-dimensional manifold
$W^\circ$, the latter being homotopy equivalent to the $(2N-1)$-sphere
$\SS^{2N-1}$. Alexander duality gives an isomorphism
$H^i_{{X'}^\circ}(W^\circ;\CC)\simeq
\Hom_\CC(H^{2N-i}_c({X'}^\circ;\CC),\CC)$ with the dual of compactly
supported cohomology, \cite[Alexander Duality V.6.6]{Iversen}. But
then ${X'}^\circ$ being a $2n$-dimensional real manifold yields by
Poincar\'e duality that $\Hom_\CC(H^{2N-i}_c({X'}^\circ;\CC),\CC)\simeq
H^{2n-2N+i}({X'}^\circ;\CC)$, \cite[I.(5.4)]{BottTu}. The latter is
$\CC$ for $i=2N-1$ and $i=2N-2n$, and zero otherwise. In the long exact
sequence
\[
\cdots\to H^i_{{X'}^\circ}(W^\circ;\CC)\to H^i(W^\circ;\CC)\to H^i(W^\circ\minus
{X'}^\circ;\CC)\stackrel{+1}{\to},
\]
we have $ H^i_{{X'}^\circ}(W^\circ;\CC)\neq 0$ only when
$i=2N-1,2(N-n)$ and $ H^i(W^\circ;\CC)\neq 0$ only if $i=2N-1,0$. The
map $\CC=H^{2N-1}_{{X'}^\circ}(W^\circ;\CC)\to
H^{2N-1}(W^\circ;\CC)=\CC$ is surjective (hence bijective) since
$W^\circ\minus {X'}^\circ$ is homotopy equivalent to an open subset of
a $(2N-1)$-sphere) and so $H^{2N-1}(W^\circ\minus {X'}^\circ;\CC)=0$.
It follows that
\begin{eqnarray}\label{eqn-*}
H^i(W\minus {X'};\CC)&=&H^i(W^\circ\minus
{X'}^\circ;\CC)=\left\{\begin{array}{ccc}\CC&\text{ if
}&i=0,2(N-n)-1;\\ 0&&\text{else.}\end{array}\right.
  \end{eqnarray}

Next we compute the cohomology of $U_W=W\minus Y', Y'=v(Y)$ where
$Y=\Var(I)$ for some homogeneous ideal $I\subseteq R_n$. Since
$U'=v(U)$ is an embedded submanifold of $U_W$ with complex codimension
$N-n$, we can consider the tubular neighborhood $T'$ of $U'$ that
arises via the tubular neighborhood theorem as the total space of the
normal bundle of $U'$ in $U_W$.
Then
\[
U_W=W\minus Y'=(W\minus X')\cup U'=(W\minus X')\cup T',
\]
with intersection $(W\minus X')\cap T'={T'}^\circ$. 

As $U',U_W$ are complex manifolds, the
removal of the zero section $U'$ from $T'$ leaves a space ${T'}^\circ$
homotopic to an oriented sphere bundle $\SS^q\into {T'}^\circ \onto U'$ where
\[
q=2(N-n)-1.
\]
The $q$-sphere bundle $T^\circ$ yields a Gysin sequence
\[
\ldots\to H^i({T'}^\circ;\CC)\stackrel{\pi_*}{\to}
H^{i-q}(U';\CC)\stackrel{e\wedge}{\to}
H^{i+1}(U';\CC)\stackrel{\pi^*}{\to} H^{i+1}({T'}^\circ;\CC)\to\ldots
  \]
Here, $\pi\colon {T'}^\circ \to U'$ is the fibration map, $\pi^*$ is the
pullback under this map,
and $e$ is the Euler class of the bundle ${T'}^\circ$ when restricted
from relative cohomology to absolute cohomology on $T$. The map
$\pi_*$ is special to the situation of bundles with fibers homotopic
to compact manifolds, and is induced by integration along the fibers
in the following sense. For any oriented $\RR^k$-bundle $E\to B$ with
$E^\circ=E\minus B$ there is a fundamental class $u\in
H^k(E,E^\circ;\ZZ)$ that restricts in each fiber to the canonical
class in $H^k(\RR^k,\RR^k\setminus\{0\};\ZZ)$; this canonical class is the
given orientation on the bundle (an orientation is a global section of
the orientation bundle with fiber
$H^k(\RR^k,\RR^k\setminus\{0\};\ZZ)$).
The existence of the fundamental class is the content of the Thom
isomorphism theorem for oriented vector bundles, and the cup product
with $u$ sets up an isomorphism $u\cup\colon H^j(E;\ZZ)\to
H^{j+k}(E,E^\circ;\ZZ)$. The cap product with the Poincar\'e dual of
$u$ induces an isomorphism $H_j(E,E^\circ;\ZZ) \to H_{j-k}(E;\ZZ)$,
the ``integration along the fibers'' above (compare
\cite[Ch.~9-12]{Milnor}). The image of $u$ in $H^k(E;\ZZ)$ is the
Euler class (by definition). If the fiber dimension $k$ is large,
$H^k(E;\ZZ)=H^{k-1}(E;\ZZ)=0$. In that case,
the Euler class of the bundle must be zero and then
$u$ corresponds to the class in
$H^{k-1}(E^\circ;\ZZ)$ with the property that it restricts in each
fiber to the canonical generator of $H^{k-1}(\RR^k\minus\{0\};\ZZ)$.

Our Gysin sequence above arises from the long exact sequence to
the pair $(T',{T'}^\circ)$ with replacements coming from the Thom
isomorphism and the fact that $U',T'$ are homotopic.

Since $\dim_\CC(U')=n$, $H^i(U';\CC)=0$ if $i\geq 2n$. On the other
hand, the Euler class is of homological degree $q+1=2(N-n)$.  Thus, if
$2(N-n)\geq 2n$ then either the source or the target of the Euler map
$H^{i-q}(U';\CC)\stackrel{e\wedge}{\to} H^{i+1}(U';\CC)$ is zero for
every $i$. But $N={n+d-1\choose n-1}\geq 2n$ for $n,d\geq 2$ unless
$d=n=2$, and usually much larger. Thus the Gysin sequence splits into
isomorphisms
\begin{align}\label{eqn-HTcirc-1}
H^i({T'}^\circ;\CC)&\stackrel{\pi_*}{\to}H^{i-q}(U';\CC)=H^{i}(T';\CC)&
\text{ if }i\geq 2n;\\ \label{eqn-HTcirc-2}
H^{i}(T';\CC)=H^i(U';\CC)&\stackrel{\pi^*}{\to}H^i({T';\CC}^\circ)&
\text{ if }i<2n.
\end{align}

Note that the composition $H^i(T';\CC)\to H^i({T'}^\circ;\CC)\to
H^i(U';\CC)$ is an isomorphism since $U'\into T'$ is a homotopy
equivalence; so the left map is an isomorphism if and only if the
right one is. Now consider the Mayer--Vietoris sequence to the pair
$(W\minus X')\cup T'=U_W$ with ${T'}^\circ=(W\minus X')\cap T'$:
\[
\cdots \to H^i((W\minus X')\cup T';\CC)\to H^i(W\minus X';\CC)\oplus
H^i(T';\CC)\to H^i({T'}^\circ;\CC)\to\cdots
\]
Here, each (component of a) map is the natural restriction, possibly
with a $(-1)$ factor.

If $i< 2n$, the map $H^i(T';\CC)\to H^i({T'}^\circ;\CC)$ in the
Mayer--Vietoris sequence is therefore the identity by
\eqref{eqn-HTcirc-2}. It follows that in this range, $H^i((W\minus
X')\cup T';\CC)\to H^i(W\minus X';\CC)$ is an isomorphism as well. But in that
range, by \eqref{eqn-*}, only $H^0(W\minus X';\CC)$ is nonzero and so
$H^i((W\minus X') \cup T';\CC)$ is zero for $0<i<2n$.

If $2n-1\le i<q$, then $H^i((W\minus X') \cup T';\CC)$ vanishes
since $H^i(W\minus X';\CC)=H^i(T';\CC)=H^i({T'}^\circ;\CC)=0$.

Let us look at the situation when 
$i=q$:
\begin{align*}
  \underbrace{H^{q-1}({T'}^\circ;\CC)}_{=H^{-1}(U';\CC)=0}\to
  H^{q}((W\minus X')\cup T';\CC)\to
  \underbrace{H^{q}(W\minus X';\CC)}_{=\CC}\oplus
  \underbrace{H^{q}(T';\CC)}_{=0}\to\qquad\\\\
  \underbrace{H^{q}({T'}^\circ;\CC)}_{=H^0(U';\CC)=\CC}\to
  H^{q+1}((W\minus X')\cup T';\CC)\to
  {\underbrace{H^{q+1}(W\minus X';\CC)}_{=0}\oplus
  \underbrace{H^{q+1}(T';\CC)}_{=0}}.
\end{align*}
If one restricts the morphism $H^q(W\minus X';\CC)\to H^q({T'}^\circ;\CC)$ to
the intersection with a small ball around a generic point of $Y'$,
both spaces become homotopic to $\SS^q$ and so the morphism
$H^q(W\minus X';\CC)\to H^q({T'}^\circ;\CC)$ restricts to an isomorphism
$\CC\to \CC$. But since $H^q(W\minus X';\CC) $ and $H^q({T'}^\circ;\CC)$ are
also equal to $\CC$, the morphism $H^q(W\minus X';\CC)\to H^q({T'}^\circ;\CC)$
is an isomorphism.  Thus, $H^{q}(U_W;\CC)=H^{q+1}(U_W;\CC)=0$.

If $i>q$, $H^i(T';\CC)=H^i(W\minus X';\CC)=0$. Thus,
$H^{i-q}(U';\CC)=H^i({T'}^\circ;\CC)=H^{i+1}((W\minus X')\cup
T';\CC)$.

We have proved
\begin{prop}\label{prop-complement-veronese}
  We use notation as defined at the start of Subsection
  \ref{subsec-veronese-complement}. Let $T'\to U'$ be the normal bundle of $U'$ in $W^\circ$.  With
  $U_W:=W\setminus Y'=(W\minus X')\cup T'$, and $q={n+d-1\choose
    n-1}>2n$ we have on the level of reduced cohomology for every
  $i\in\ZZ$ the isomorphisms
  \[
  \xymatrix{
    H^i(U';\CC) \ar[r]^{{\pi^*}}_\simeq &
    H^i(T';\CC) \ar[dr]_{e\cup}^\simeq
    \ar[r]^{e_0\cup}_\simeq& H^{i+q}({T'}^\circ;\CC) \ar[d]^{\delta^*}_\simeq \ar[r]^{\delta^*}_\simeq&
    H^{i+q+1}(U_W;\CC)
    \\
    &&H^{i+q+1}(T',{T'}^\circ;\CC)       }
  \]
  Here $e\in H^{q+1}(T',{T'}^\circ;\ZZ)$ is the Euler class of the
  bundle, $e_0\in H^q({T'}^\circ;\ZZ)$ is its preimage, the vertical
  $\delta^*$ is the connecting morphism for the pair
  $(T',{T'}^\circ)$, and the horizontal $\delta^*$ is the connecting
  morphism for the Mayer--Vietoris spectral sequence for the cover
  $U_W=(W\minus X')\cup T'$.

  In particular, the singular reduced cohomology groups of the
  complements of the cones $Y$ and $Y'$ over $\tilde Y$ are the same
  up to a cohomological shift by $q+1=2(N-n)$.
\end{prop}

\begin{cor}
  If $Y\subseteq X=\CC^n$ is homogeneous and of equi-dimension three,
  then the \v Cech--de Rham numbers $\rho_{p,q}^r(Y)$ are invariant
  under Veronese maps of $Y$.
\end{cor}
\begin{proof}
 With $Y'\subseteq W$ and notation as in Subsection
 \ref{subsec-veronese-complement}, 
    the \v Cech--de Rham spectral sequence
  degenerates for dimensional reasons by Proposition
  \ref{prop-cdR-dim3}.  According to Proposition
  \ref{prop-complement-veronese}, the two complements have the same
  reduced cohomology up to a shift by the relative dimension. Hence, 
  up to that same shift, the two \v Cech--de Rham spectral sequences
  have the same abutment. The degeneration shows that the abutment
  determines the $\rho^2_{p,q}$, except for the numbers
  $\rho^2_{2,2}+\rho^2_{3,0}=\dim H^{n-3}_\dR(U)$. However, $\rho^2_{2,2}$
  is the dimension of $H^{n-2}_\dR(H^{n-2}_I(R))$ and thus equals the socle
  dimension of $H^2_\frakm(\DD H^{n-2}_I(R))$ by Lemma \ref{lem-rw}. But
  equi-dimensionality and Remark \ref{rmk-L-argument} show that
  $H^2_\frakm(\DD H^{n-2}_I(R))=0$.  This settles the case $r=2$. But no
  higher nonzero differentials can exist by degeneration.
\end{proof}

It turns out that, similarly to the corresponding result on Lyubeznik
numbers in \cite{RSW}, the \v Cech--de Rham numbers of level two do
not change under Veronese maps of projective varieties.

\begin{thm}
  Let $\tilde Y$ be a projective variety and suppose $Y$ is a
  cone for $\tilde Y$, embedded into an affine space $X$. Then the
  \v Cech--de Rham numbers $\rho^2_{k,\ell}$ do for $k\geq 2$ not
  depend on $Y$ but are a function of $\tilde Y$ alone.
\end{thm}
The overall plan of the proof is quite similar to the proof in
\cite{RSW} that the $\lambda_{k,\ell}$ are (largely) unchanged under
Veronese maps. We start with translating the $\rho_{k,\ell}$ into
objects of constructible sheaves involving the Verdier dual $\omega_Y$
of the constant sheaf on $Y$. After some rewriting we use Lemma
\ref{lem-rw} 
 to exchange a direct image functor to a point for the
pullback to the origin and then to use an adjunction triangle to
reformulate them in terms of $Y^\circ$. We finally lift to $\tilde Y$ on which
one uses an interpretation in terms of Chern classes.

During the proof we shall use the following diagram of maps
\[
\xymatrix{\{0\} \ar[d]_{i_Y} \ar@{=}[r] & \{0\} \ar[d]^i  \\ Y \ar@/_1.5em/[u]_{a_Y} \ar[r]^h & \CC^{n}=:X \ar@/_1.5em/[u]_{a_X}  \\
Y_0 \ar[u]^{j_Y} \ar[d]_\pi \ar[r]^-{h_0} &\CC^{n} \setminus \{0\}=: X^\circ \ar[u]_j \ar[d]^p& \\
\tilde Y \ar[r]^g & \PP^{n-1} }
\]

\begin{proof}

The \v Cech--deRham numbers of level 2 are given by 
\[
\rho_{k,\ell} :=\dim  H^{n-k}_{\dR} H^{n-\ell}_I R  = \dim {^p}\calH^{-k} a_*({^p}\calH^{n-\ell} h_! h^! \underline{\CC}_{X}[n]).
\]
(The shift from $-n+k$ on the left to $-k$ on the right occurs since
the de Rham functor used in the
Riemann--Hilbert correspondence arises from the ``natural'' algebraic
de Rham functor---which goes along with the tensor product with
$\omega$---by analytification \emph{and} a shift by $n$).

We have
\begin{eqnarray}\label{eqn-TR-1}
  {^p}\calH^{-k} a_*({^p}\calH^{n-\ell} h_! h^! \underline{\CC}_{X}[n])
& \stackrel{({\rm a})}{\simeq}& {^p}\calH^{-k} a_*({^p}\calH^{n-\ell} h_! h^! a_X^! \CC_{\pt}[-n])\notag\\
& \stackrel{({\rm b})}{\simeq}& {^p}\calH^{-k}  a_*({^p}\calH^{n-\ell} h_! \omega_Y[-n]) \notag\\
& \stackrel{({\rm c})}{\simeq}& {^p}\calH^{-k} a_* h_! ({^p}\calH^{n-\ell}  \omega_Y[-n])\notag\\
& \stackrel{({\rm d})}{\simeq}& {^p}\calH^{-k} a_* h_* ({^p}\calH^{n-\ell}  \omega_Y[-n])\\
& {\simeq}& {^p}\calH^{-k}  (a_Y)_* ({^p}\calH^{n-\ell}  \omega_Y[-n])\notag\\
& {\simeq}& {^p}\calH^{-k}  (a_Y)_* ({^p}\calH^{-\ell}  \omega_Y)\notag\\
& \stackrel{({\rm e})}{\simeq}& {^p}\calH^{-k}  (i_Y)^{-1} ({^p}\calH^{-\ell}  \omega_Y)\notag
\end{eqnarray}
The justfications are as follows: (a) holds since the real dimension of $X$
is $2n$; (b) follows from the definition of $\omega_Y$; (c) holds since $h$ is a
closed embedding and hence $h_!$ is perverse exact;
(d) comes from $h_*=h_!$ for closed
embeddings; (e) is Lemma 3.3 in \cite{RW-weight}.

We have the following triangle from the inclusion of the origin into $Y$:
\[
 j_{Y!} j_Y^{-1} ({^p}\calH^{-\ell}  \omega_Y) \to  {^p}\calH^{-\ell}  \omega_Y \to  i_{Y!} i_Y^{-1} ({^p}\calH^{-\ell}  \omega_Y) \overset{+1}\to
\]
and it induces the following long exact sequence
\[
\xymatrix{  &  \ldots \ar[r] & {^p}\calH^{-3} i_{Y!} (i_Y)^{-1} ({^p}\calH^{-\ell}  \omega_Y) \ar[dll] \\ 
{^p}\calH^{-2} j_{Y!} (j_Y)^{-1} ({^p}\calH^{-\ell}\omega_Y) \ar[r] &  0 \ar[r] & {^p}\calH^{-2} i_{Y!} (i_Y)^{-1} ({^p}\calH^{-\ell}  \omega_Y) \ar[dll] \\ 
 {^p}\calH^{-1} j_{Y!} (j_Y)^{-1} ({^p}\calH^{-\ell}\omega_Y) \ar[r] & 0 \ar[r] & {^p}\calH^{-1} i_{Y!} (i_Y)^{-1} ({^p}\calH^{-\ell}  \omega_Y) \ar[dll] \\
 {^p}\calH^0 j_{Y!} (j_Y)^{-1} ({^p}\calH^{-\ell}\omega_Y) \ar[r] & {^p}\calH^{-\ell} \omega_Y \ar[r] & {^p}\calH^0 i_{Y!} (i_Y)^{-1} ({^p}\calH^{-\ell} \omega_Y) \ar[r] & 0}
\]

We now use that $k\geq 2$, which yields the following isomorphisms
from the long exact sequence above:
\begin{eqnarray}\label{eqn-TR-2}
{^p}\calH^{-k} i_{Y!} (i_Y)^{-1} ({^p}\calH^{-\ell}\omega_Y) &\simeq& {^p}\calH^{-k+1} j_{Y!}(j_Y)^{-1} ({^p}\calH^{-\ell} \omega_Y) \notag\\
&\stackrel{({\rm f})}{\simeq}& {^p}\calH^{-k+1} j_{Y!}({^p}\calH^{-\ell} \omega_{Y^\circ}) \notag\\
&\stackrel{({\rm g})}{\simeq}& {^p}\calH^{-k+1} j_{Y!}({^p}\calH^{-\ell} \pi^!
\omega_{\tilde Y})\\
&{\simeq}&  {^p}\calH^{-k+1} j_{Y!}({^p}\calH^{-\ell+1} \pi^![-1]
\omega_{\tilde Y})\notag\\
&\stackrel{({\rm h})}{\simeq}&  {^p}\calH^{-k} j_{Y!}\pi^!({^p}\calH^{-\ell+1}
\omega_{\tilde Y})\notag
\end{eqnarray}
with justifications as follows:
(f) since $j_Y$ is open and so $(j_Y)^{-1}$ is perverse exact; (g) is
dual to the fact that $\CC_{Y^\circ}=\pi^{-1}\CC_{\tilde{Y}}$;
(h) is because $\pi$
is smooth so that $\pi^![-1]$ is perverse exact.

We have then
\begin{align*}
{^p}\calH^i a_{Y!} {^p}\calH^{-k} (j_{Y!}\pi^!({^p}\calH^{-\ell+1}
\omega_{\tilde Y}))  \simeq 
{^p}\calH^{i} a_{Y!}({^p}\calH^{-k} (i_{Y!} (i_Y)^{-1} ({^p}\calH^{-\ell}\omega_Y))) = 0 \qquad \text{for}\quad  i \neq 0
\end{align*}
since ${^p}\calH^{-n+k} i_{Y!} (i_Y)^{-1} ({^p}\calH^{-\ell}\omega_Y)$ is at
most supported on a point. A spectral sequence argument shows
therefore that
\begin{gather}\label{eqn-TR-3}
{^p}\calH^0 a_{Y!} {^p}\calH^{-k} j_{Y!}\pi^!({^p}\calH^{-\ell+1}  \omega_{\tilde{Y}})  \simeq {^p}\calH^{-k} a_{Y!} j_{Y!}\pi^!({^p}\calH^{-\ell+1}  \omega_{\tilde{Y}}).
\end{gather}
Summarizing we have for $k\geq 2$ that
\begin{eqnarray}\label{eqn-TR-4}
{^p}\calH^{-k} a_* ({^p}\calH^{n-\ell} h_!h^! \underline{\CC}_{V}[n])
&\stackrel{({\rm i})}{\simeq}& {^p}\calH^{-k}  (i_Y)^{-1} ({^p}\calH^{-\ell}  \omega_Y)\notag\\
&\stackrel{({\rm j})}{\simeq}& {^p}\calH^0 a_{Y!} i_{Y!} ({^p}\calH^{-k}  (i_Y)^{-1} ({^p}\calH^{-\ell}  \omega_Y))\\
&\stackrel{({\rm k})}{\simeq}& {^p}\calH^0 a_{Y!} ({^p}\calH^{-k}  i_{Y!} (i_Y)^{-1} ({^p}\calH^{-\ell}  \omega_Y))\notag\\
&\stackrel{({\rm l})}{\simeq}& {^p}\calH^{-k} a_{Y!} j_{Y!}\pi^!({^p}\calH^{-\ell+1}  \omega_{\tilde{Y}}),
\end{eqnarray}
where (i) follows from display \eqref{eqn-TR-1}; (j) follows since
$a_Y\circ i_Y$ is the identity on $\{0\}$; (k) is
since $i_{Y!}$ is perverse exact as $i_Y$ is closed; (l) comes from
displays 
\eqref{eqn-TR-2} and \eqref{eqn-TR-3}. 


Now let $\calL$ be the quasi-coherent pullback of
$\calO_{\PP^{n-1}}(1)$ via $g$. By abuse of notation we denote the
total space of the corresponding line bundle  by the same letter. Notice that $Y^\circ \simeq \calL \minus \{\text{zero section}\}$. Consider the following diagram
\[
\xymatrix{\{0\} \ar[d]_{i_Y} &  \tilde Y \ar[l] \ar[d]^{\tilde{i}} \ar@{=}[dr] & \\
Y & \calL \ar[l]_u \ar[r]^q & \tilde Y\\
Y^\circ \ar[u]^{j_Y} \ar@{=}[r] & Y^\circ \ar[u]_{\tilde{j}} \ar[ur]_\pi &}
\]
in which $q$ is the bundle map, $\tilde i$ the embedding of the
zero section, and $u$ is the contraction of the zero section.
We have
\begin{align*}
{^p}\calH^{-k} a_{Y!} j_{Y!}\pi^!({^p}\calH^{-\ell+1}  \omega_{\tilde
  Y})&\simeq {^p}\calH^{-k} a_{Y!} j_{Y!
}\pi^{-1}[2]({^p}\calH^{-\ell+1}  \omega_{\tilde Y})\\ 
&\simeq{^p}\calH^{-k+2} a_{Y!} j_{Y!	}\pi^{-1}({^p}\calH^{-\ell+1}
\omega_{\tilde Y})\\
&\simeq {^p}\calH^{-k+2} a_{Y!}
j_{Y!}(\widetilde{j})^{-1}q^{-1}({^p}\calH^{-\ell+1}  \omega_{\tilde Y}) \\
&\simeq {^p}\calH^{-k+2} a_{X!} q_{!}\ \widetilde{j}_!
(\widetilde{j})^{-1}q^{-1}({^p}\calH^{-\ell+1}  \omega_{\tilde Y})\\
&\simeq {^p}\calH^{-k+2} a_{X!} \pi_! \pi^{-1}({^p}\calH^{-\ell+1}
\omega_{\tilde Y})\\
&\simeq\HH^{-k+2}_c(X,\pi_! \pi^{-1}({^p}\calH^{-\ell+1}  \omega_{\tilde
Y}))
\end{align*}
Here the first isomorphism comes from the fact that
$\pi^{-1}[1]=\pi^![-1]$ is perverse exact, and the last one because
cohomology of compact supports is the cohomology of the exceptional
direct image functor.

From the closed embedding of  $\tilde Y$ into $\calL$ arises a triangle
\begin{equation}
 \widetilde{j}_! (\widetilde{j})^{-1}q^{-1}({^p}\calH^{-\ell+1}  \omega_{\tilde{Y}}) \to q^{-1}({^p}\calH^{-\ell+1}  \omega_{\tilde{Y}}) \to  \tilde{i}_! (\tilde{i})^{-1} q^{-1}({^p}\calH^{-\ell+1} \omega_{\tilde{Y}}) \overset{+1}\to
\end{equation}
Applying $q_!$ we get
\[
\pi_! \pi^{-1} \calG_{-\ell+1} \to q_!q^{-1} \calG_{-\ell+1} \to \calG_{-\ell+1}  \overset{+1}\to
\]
where we have set $\calG_{-\ell} := {^p}\calH^{-\ell} \omega_{\tilde{Y}}$ and used
$\pi=q\circ\tilde j$.
We have $\calG \simeq q_! q^!\calG \simeq q_!q^{-1}\calG[2]$ for any
$\calG \in \textup{Perv}(\tilde{Y})$ since $q$ is smooth of relative dimension
$1$. This gives the triangle
\[
\pi_! \pi^{-1} \calG_{-\ell+1}  \to \calG_{-\ell+1}[-2] \to \calG_{-\ell+1}  \overset{+1}\to
\]
As in  \cite[(1.3.1)]{RSW}, this triangle is dual to a triangle $\calF \to \calF[2] \to p_*\pi^!\calF \overset{+1}\to$ where the first map
is induced by
\[
e \otimes 1 : \CC_{\tilde{Y}} \otimes \calF \to \CC_{\tilde{Y}}[2] \otimes \calF,
\]
with $e \in \Hom_{D^b_\cs({\tilde{Y}})}(\CC_{\tilde{Y}},\CC_{\tilde{Y}}[2]) \simeq
\Hom_{D^b_\cs(\pt)}(\CC,R\Gamma({\tilde{Y}};\CC_{\tilde{Y}}[2])) \simeq H^2({\tilde{Y}};\CC)$ is the image
of the Euler class of the vector bundle $\calL$.

We get a long exact sequence
\[
\to \HH^{-k+2}_c({\tilde{Y}},\pi_! \pi^{-1} \calG_{-\ell+1}) \to  \HH^{-k}_c({\tilde{Y}},\calG_{-\ell+1})(-1) \to \HH^{-k+2}_c({\tilde{Y}},\calG_{-\ell+1}) \to \HH^{-k+3}_c({\tilde{Y}},\pi_! \pi^{-1} \calG_{-\ell+1}) \to
\]

In particular we get short exact sequences
\[
0\to \HH^{-k+1}_c({\tilde{Y}},\calG_{-\ell+1})_\calL \to  \HH^{-k+2}_c({\tilde{Y}},\pi_! \pi^{-1} ({^p}\calH^{-\ell+1} \omega_{\tilde{Y}})) \to \HH^{-k}_c({\tilde{Y}},\calG_{-\ell+1})^\calL \to 0
\]
where
\begin{eqnarray*}
\HH^{-k+1}_c({\tilde{Y}},\calG_{-\ell+1})_\calL &:=& \textup{coker}\left( \HH^{-k-1}_c({\tilde{Y}},\calG_{-\ell+1})(-1) \to \HH^{-k+1}_c({\tilde{Y}},\calG_{-\ell+1})\right)\\
\HH^{-k}_c({\tilde{Y}},\calG_{-\ell+1})^\calL &:=& \ker\left(\HH^{-k}_c({\tilde{Y}},\calG_{-\ell+1})(-1) \to \HH^{-k+2}_c({\tilde{Y}},\calG_{-\ell+1})\right)
\end{eqnarray*}
  
Putting everything together we get that
\[
\rho_{k,\ell} = \dim \HH^{-k+2}_c({\tilde{Y}},\pi_! \pi^{-1}({^p}\calH^{-\ell+1}
\omega_{\tilde{Y}})) = \dim \HH^{-k+1}_c({\tilde{Y}},\calG_{-\ell+1})_\calL + \dim
\HH^{-k}_c({\tilde{Y}},\calG_{-\ell+1})^\calL
\]
is unchanged under Veronese maps for $k \geq 2$, since the Euler class of
a bundle power is a multiple of the original Euler class (and so
over $\CC$  kernels and cokernels are preserved).

\end{proof}

\section{Lyubeznik numbers}

In this section we study the Lyubeznik numbers and their spectral
sequence. After surveying some known facts we discuss to what extent a
projective variety determines the Lyubeznik numbers of its cone(s). We
look first specifically at varieties of Picard number 1, listing some
examples and open questions. After that we discuss cases where in
small dimension the Lyubeznik tables of all cones agree.

\subsection{Basic properties}

We should begin with drawing some parallels to the case of the \v
Cech--de Rham numbers. Quite immediately, being defined as the socle
dimensions  of
the $E_2$-terms in the Grothendieck spectral sequence
\eqref{eqn-lc-ss}, the Lyubeznik numbers vanish for
$q\not\in [\codim(I),n]$ and for $p\not\in [0,n]$. In fact, 
similarly to the $\rho^r_{p,q}$, Lyubeznik numbers fit into a
triangular region
\[
\Lambda(Y):=\begin{pmatrix}\lambda_{0,0}&\cdots&\cdots&\lambda_{0,d}\\
  0&\ddots&&\vdots\\
  \vdots&\ddots&\ddots&\vdots\\
  0&\cdots&0&\lambda_{d,d}\end{pmatrix}.
\]
In this
picture, the differentials of the Grothendieck spectral sequence point
South to Southeast.

\begin{rmk}
  \label{two zero terms}
  The fact that the abutment is $H^n_\frakm(R)$ implies that the
  entries
\[
\lambda_{0,d}=\lambda_{1,d}=0
\]
always vanish unless the dimension of $I$ is less than two, in which
cases the Lyubeznik tables are $(1)$ and
$\begin{pmatrix}\cdot&\cdot\\\cdot&1\end{pmatrix}$ respectively.
\schluss
\end{rmk}
  
The number $\lambda_{d,d}$ is never zero by \cite{L-Dmods} and related
to connectedness issues. For example,  
if $\dim(Y)=2$ then $\Lambda=\begin{pmatrix}
\cdot&a-1& \cdot\\
\cdot &\cdot &\cdot\\
\cdot &\cdot&a\end{pmatrix}$ where $a$ is
the number of connected components of the punctured spectrum of the ring
defining the purely 2-dimensional part of $I$,
\cite{W-lambda,Kawasaki}. By \cite{Zhang-Compositio}, $\lambda_{d,d}$
is the number of connected components of the Hochster--Huneke graph of
the completed strict Henselization of $R/I$.

It was first observed in \cite{GarciaLopez-Sabbah} that the Lyubeznik
numbers encode interesting topological information also in higher
dimension. However, it is often not easy to decode this
information. Garcia and Sabbah concentrate on the case of an isolated
singularity and find that the topology of the singularity link carries all
information on $\Lambda$. Other relations to connectedness
dimensions are discussed in the survey \cite{BWZ-survey}.

A new angle was introduced in \cite{RSW} by applying the theory of
perverse sheaves and mixed Hodge modules to the problem. Indeed, over
$\CC$, the local cohomology groups $H^\bullet_I(R)$ are,  as
the pushforward of the structure sheaf on $U(I)$ to $\CC^n$, all
equipped with a natural mixed Hodge module structure. The corresponding
perverse sheaves carry information on the intersection homology of
$U(I)$. Using this connection it is proved in \cite{RSW} that if $I$ is
homogeneous then the $\lambda_{p,q}(R/I)$ can be recovered from the
kernel and cokernel of certain maps of sheaves on the corresponding
projective scheme, at least as long as $p>1$.

\subsection{Lyubeznik numbers and projective schemes}

Suppose $\tilde Y$ is a projective variety in $\PP^{n-1}_\KK$, with
defining ideal $I\subseteq R_n=\KK[x_1,\ldots,x_n]$. Different
embeddings of $\tilde Y$ give rise to different ideals in different
polynomial rings, and thus potentially to different sets of Lyubeznik
numbers.  That this is indeed a possibility was shown to be the case
in \cite[Sections 2.2, 2.3]{RSW} where a projective variety with two
embeddings is constructed that produce (partially) different
$\lambda_{p,q}$; see also \cite{Botong-PAMS}.
On the other hand, if $\tilde Y$ is smooth or an
$\QQ$-homology manifold or analytically locally a set-theoretic
complete intersection, then all cones for $\tilde Y$ yield the same
Lyubeznik numbers \cite{GarciaLopez-Sabbah,Switala-proj-lambda,RSW}.

That examples of cones over $\tilde Y$ with varying Lyubeznik numbers
exist over $\CC$ is rather surprising at first, since similar examples
cannot exist in any positive characteristic. Indeed, it is shown in
\cite{Zhang-Advances} that Lyubeznik numbers in finite characteristic
can be seen as eigenvalues of certain operators on sheaves that are
intrinsic to the projective variety associated to $I$.

An interesting feature of \cite{RSW} is the realization that the
non-vanishing Lyubeznik numbers of a homogeneous ideal can, for $p>1$,
be viewed as a measure for the failure of a certain morphism of
perverse sheaves to be an isomorphism. All known examples of
projective varieties with possibly varying Lyubeznik numbers of their
cones 
come from varieties with Picard number at least two. An
application of \cite[Prop.~3]{RSW} is that most Lyubeznik numbers are
unchanged under Veronese maps.

\begin{lem}\label{lem-consOfPic}
  Let $Y$ be a cone over the projective variety $\tilde Y\subseteq
  \PP^{n-1}_\CC$. Let $v_n^d$ be the $d$-th Veronese applied to the
  cone $Y$, and write $Y'=v_n^d(Y)$ for the new cone. Then for $p\geq 2$, the
  Lyubeznik numbers $\lambda_{p,q}(Y)$ and $\lambda_{p,q}(Y')$ agree.

  In particular, if the Picard number $\tilde Y$ equals
  one, then the Lyubeznik numbers $\lambda_{\ge 2,q}(Y)$ to cones over
  $\tilde Y$ are independent of the cone.
  \end{lem}
\begin{proof}
  Let $\iota_1,\iota_2$ be two embeddings of $\tilde Y$ into
  projective spaces $\PP^{n-1}_\KK, \PP^{m-1}_\KK$ and denote
  $Y_1\subseteq X_1,Y_2\subseteq X_2$ the two cones over $\tilde Y$,
  sitting in the respective affine spaces that belong to the two
  embeddings. Let $\calL_1,\calL_2$ be the associated line bundles on
  $\tilde Y$ obtained as pullbacks of $\calO_{\PP^{n-1}_\KK}(1)$ and
  $\calO_{\PP^{m-1}_\KK}(1)$ respectively. Then by
  \cite[Prop.~1,2,3]{RSW}, the Lyubeznik numbers $\lambda_{p,q}$ of
  $\tilde Y$ that belong to $Y_i$ and have $p\geq 2$ are determined by
  the (co)kernel sizes of the Chern classes of $\calL_i$ on certain
  cohomology groups of $\tilde Y$ with rational coefficients. These
  cohomology groups themselves (see \cite[Prop.~2]{RSW}) do not depend
  on the bundles $\calL_i$.

  If $\iota_2$ is the $d$-fold Veronese applied to $\iota_1$ then the
  (first) Chern class of $\calL_2$ is $d$ times that of $L_1$. In
  particular, their kernels and cokernels on $\QQ$-spaces are
  identical and the first claim follows.

  Now suppose that the target of the natural map
  \[
  \phi\colon \Pic(\tilde Y)\to \Pic(\tilde Y)\otimes_\ZZ \QQ
  \]
  is $\QQ$.  If $\iota_1,\iota_2$ are both projective embeddings of
  $\tilde Y$, ampleness implies that $q(\calL_i)>0$.
  Then if $\phi(\calL_1)=q_1$ and $\phi(\calL_2)=q_2$, both
  positive rational numbers, we have for $k\gg 0$ with $kq_1,kq_2\in\NN$ that
  $\calL_1^{k|q_2|}=\calL_2^{k|q_1|}$.
  Then
  by the first part of the proof,
  $\iota_1$, $\iota_1^{kq_2}$, $\iota_2^{kq_1}$ and $\iota_2$ all
  yield the same Lyubeznik numbers $\lambda_{p,q}$ for $p\geq 2$
  (where we write $\iota^\ell$ for the $\ell$-th Veronese of the
  embedding $\iota$).
\end{proof}

For $p<2$, we do not know how to compare the $\lambda_{p,q}$ of different cones. 
\begin{prb}
  Is it true that if the Picard number of $\tilde Y$ is one, then all
  Lyubeznik numbers of all cones $Y$ of $\tilde Y$ agree?\schluss
  \end{prb}

Here are three interesting sets of varieties
to which this lemma applies.

\subsubsection{Determinantal ideals}

\begin{prop}
  The Lyubeznik numbers $\lambda_{p,q}$ with $p\geq 2$ of (the cones
  over) the projective determinantal varieties $\tilde Y_{m,n,t}$ cut out by
  the $t\times t$ minors of an $m\times n$ matrix of indeterminates
  are unique.
\end{prop}
\begin{proof}
  Let $A_{m,n,t}$ be the ring obtained as quotient of the polynomial
  ring $\KK[x_{i,j}|1\le i\le m,1\le j\le n]$ by the $t$-minors of the
  matrix $x:=((x_{i,j}))$.

  
  The case $t=1$ is trivial. If $t=2$, the associated projective
  variety is the product of two projective spaces, and in particular
  smooth. By \cite{GarciaLopez-Sabbah}, or \cite{Switala-proj-lambda}, the
  Lyubeznik numbers of $\tilde Y_{m,n,2}$ are independent of the embedding.

  Now consider the case $t>2$.  By \cite[Cor.~8.4]{BrunsVetter}, the
  divisor class group of $A_{m,n,t}$ is $\ZZ$, a generator being the
  ideal $\bar I_{m,n,t-1}$ of $A_{m,n,t}$ generated by the
  $(t-1)$-minors of the first $t-1$ rows (or columns) of $x$.


  Since determinantal varieties are normal, they satisfy 
  condition $(*)$ in \cite[Page 130]{Hartshorne-book}.  By
  \cite[Exercise II.6.3]{Hartshorne-book}, there is a short exact
  sequence $0\to \ZZ\to \Cl(\tilde Y)\to\Cl(Y)\to 0$, $Y$ the cone
  over $\tilde Y$, where the last map factors through the class group
  $\Cl(Y\minus
  P)$ of the complement of the origin in $Y$.
  For $t\geq 2$ this implies that $\Cl(\tilde
  Y_{m,n,t})=\ZZ\oplus\ZZ$. In this sequence, $1\in\ZZ$ is sent to the
  generic hyperplane section of $\tilde Y$. In order to determine the
  Picard group of $\tilde Y_{m,n,t}$ we need by
  \cite[Prop.~II.6.15]{Hartshorne-book} to determine the Cartier
  classes of $\Cl(\tilde Y_{m,n,t})$. From the preceding, this amounts
  to checking which multiples of $\bar I_{m,n,t-1}$ are Cartier on the
  punctured spectrum of $A_{m,n,t}$. One sees easily that for $t=2$,
  $\bar I_{m,n,t-1}$ is Cartier on the punctured spectrum. For $t>2$
  only its trivial power is Cartier: by the coordinate change
  expounded in \cite{LSW}, powers of $\bar I_{m,n,t-1}$ are locally
  principal on the open set $U_{x_{1,1}}$ if and only if corresponding
  powers of $\bar I_{m-1,n-1,t-2}$ are locally principal everywhere on
  $Y_{m-1,n-1,t-1}$; for $t=3$ this is clearly not so.  Hence the
  Picard group of $\tilde Y_{m,n,t}$ is $\ZZ$ for $t>2$. Now use Lemma
  \ref{lem-consOfPic}. 

\end{proof}

\begin{rmk}
  In particular, the Lyubeznik numbers $\lambda_{p,q}$ of
  determinantal varieties computed by L\"orincz and Raicu in \cite{LR}
  for the standard embedding equal those of any embedding, at least for
  $p\geq 2$.
\schluss
\end{rmk}

\begin{rmk}
  Suppose $G$ is a semisimple linear algebraic group, $P$ a parabolic
  subgroup and $w$ an element of the Weyl group of $G$. The Schubert variety
  $X_P(w):=BwP/P$ sits inside the homogeneous space $G/P$, and every
  line bundle on $X_P(w)$ is the restriction of a line bundle on $G/P$,
  \cite{Mathieu}. In particular, the Picard group of $X_P(w)$ is
  generated by the Schubert divisors (the Schubert varieties inside
  $X_P(w)$ of codimension one), \cite[Prop.~2.2.8]{Brion}.\schluss
\end{rmk}

\begin{prb}
  Compute the Lyubeznik numbers of Schubert varieties in their
  Schubert divisor embeddings.\schluss
\end{prb}

\subsubsection{Toric varieties}

Suppose $\tilde Y$ is the toric variety attached to a complete fan
$\Delta$ that is projective. If $\Delta$ is smooth, or at least
simplicial, then the Picard group of $\tilde Y$ is a free Abelian
group generated by the torus invariant (Cartier) divisors
corresponding to the $n$ rays of $\Delta$,
\cite[Thm.~4.2.1]{CoxLittleSchenck}. The ambient lattice imposes
$d:=\dim(\tilde Y)$ many independent relations on these divisors, so
that $\Pic(\tilde Y)=\ZZ^{n-d}$. In order for this number to be 1,
there is very little choice for $\Delta$; it forces $\tilde Y$ to be a
weighted projective space. These are $\QQ$-homology manifolds
and thus yield the same Lyubeznik numbers under all embeddings by
\cite{RSW}.

However, singular fans fail the Picard rank formula above and can have
Picard group $\ZZ$ with greater variety. The Picard group is
free if the fan is full-dimensional by
\cite[Thm.~4.2.5]{CoxLittleSchenck}, and equals the inverse limit of
the quotient lattices $M/M(\sigma)$, taken modulo $M$ by
\cite[Thms.~4.2.1,4.2.9]{CoxLittleSchenck}. 

\begin{exa}
If $\Delta$ is a complete rational fan in $\ZZ^3$, one can use the
description of the Picard group via support functions
to show that if $\Delta$ has at most one
simplicial cone, then the Picard group of the associated toric variety
is rank one. For example, the fan over the sides of a cube leads to a
projective three-fold with Picard number one. The generating support
function takes the value zero on one square and one on the opposing
square (see
\cite[Exa. 1.5.(3)]{Fulton}).
Our next result shows that all projective toric threefolds
have their Lyubeznik table independent of the embedding. 
\end{exa}

\begin{thm}\label{thm-toric-3fold}
  Let $\tilde Y$ be the projective variety to a complete projective
  fan in $\ZZ^3$ with Picard number $\tilde p+1$. Then for any cone $Y$ over
  $\tilde Y$ its Lyubeznik numbers take the form
  \[
  \Lambda(Y)=\begin{pmatrix}
    \cdot&\cdot&\cdot&\tilde p&\cdot\\
    \cdot&\cdot&\cdot&\cdot&\cdot\\
    \cdot&\cdot&\cdot&\cdot&\tilde p\\
    \cdot&\cdot&\cdot&\cdot&\cdot\\
    \cdot&\cdot&\cdot&\cdot&1\end{pmatrix}
    \]
\end{thm}
We postpone the proof until the end of the final section. 

\begin{prb}
  Express Lyubeznik numbers of projective toric varieties (of Picard rank 1
  or otherwise)
  in terms of fan data (and, possibly, polytopes).\schluss
\end{prb}

\subsubsection{Horospherical varieties}

Horospherical varieties are complex normal algebraic varieties on
which a connected complex reductive algebraic group $G$ acts with an
open orbit that is isomorphic to a torus bundle over a flag
variety; the dimension of this torus is referred to as the rank of the
variety. In particular, toric and flag varieties are examples of
horospherical varieties.

Any flag variety $G/P$ with $P$ a parabolic subgroup of $G$ is
smooth and projective. Their Lyubeznik numbers
are hence all topological, by \cite{Switala-proj-lambda}.  By
\cite[Thm.~0.1]{Pasquier-Picard1}, a smooth projective horospherical
variety of Picard number 1 must either be a homogeneous space or have
horospherical rank one.

There are many
singular horospherical  varieties of Picard number one.
For example,
let $G$ be a simple linear algebraic group and choose two dominant weights
$\chi_1$ and $\chi_2$ that cannot be written as the sum of a common
dominant weight with another dominant weight. Writing $V(\chi)$ for the
simple $G$-module of weight $\chi$, 
let $\tilde Y$ be the closure of the $G$-orbit of the sum of two
highest weight vectors in $\PP(V(\chi_1)\oplus V(\chi_2))$.  It is a
projective variety of horospherical rank one, it has Picard number one
and is smooth only in very few cases, namely when $\chi_1$ and
$\chi_2$ are fundamental weights $\varpi_\alpha$ and $\varpi_\beta$
and $(G,\alpha,\beta)$ is in the list of
\cite[Thm.~1.7]{Pasquier-Picard1}. It has three G-orbits (one open and
two closed), the singularities if they exist, are on the closed
orbit(s). By taking a longer list of weights $\chi_1,\ldots,\chi_n$
one can produce (usually singular) varieties of horospherical rank
$n-1$.

In all these Picard rank 1 cases, the Lyubeznik numbers
$\lambda_{p,q}$ with $p>1$ of the cone of $\tilde Y$ are embedding independent,
and can hence be computed from the embedding that arises from the
definition.
\begin{prb}
  Compute Lyubeznik numbers of horospherical varieties of Picard number
  one for the standard families.\schluss
\end{prb}

\subsection{Lyubeznik numbers in small dimension}

We consider here to what extent the Lyubeznik numbers of varieties of
small dimension are functions of the associated projective variety
only. Some independence is known quite generally.

\begin{rmk}
\label{highest lambda}
\begin{asparaenum}

\item If $\tilde Y$ is a projective scheme of dimension at most $1$ over any
  field (not necessarily connected or equi-dimensional) then $R/I$ is
  two-dimensional for any embedding, and so $\Lambda$ is independent
  of embeddings by \cite{W-lambda,Kawasaki}

\item Set $d-1=\dim(\tilde Y)$. Then $\lambda_{d,d}$ is independent of
  embeddings unconditionally by \cite{Zhang-Compositio}.
  \schluss

\end{asparaenum}
\end{rmk}

We begin with some preparations involving Hartshorne's (local) algebraic de
Rham cohomology.
\begin{thm}
\label{thm: main}
Let $\tilde Y\subseteq \PP^{n-1}_\KK$ be a projective variety over a field
$\KK$ of characteristic 0, let $Y\subseteq \AA^{n}_\KK$ be the
affine cone of $\tilde Y$, and let $P$ be its vertex. Let $H^j_P(Y)$ denote
the local de Rham cohomology of $Y$ supported in $\{P\}$. Assume that
the Picard group of $\tilde Y$ has rank 1. Then $\dim_k(H^j_P(Y))$
depends only on $\tilde Y$, but not on the embedding $\tilde
Y\subseteq \PP^n_\KK$. More precisely, if $\tilde Y \subseteq \PP^{n'-1}_\KK$ is
another embedding of $\tilde Y$ into a projective space and $Y'$ is its
affine cone with the vertex $P'$ , then
$\dim_\KK(H^j_P(Y))=\dim_\KK(H^j_{P'}(Y'))$.
\end{thm}

To prove Theorem \ref{thm: main}, we need the following result of Hartshorne.
\begin{thm}[Proposition III.3.2 in \cite{Hartshorne-DRCAV}]
Let $\tilde Y,Y,P$ be the same as in Theorem \ref{thm: main}. Then
$H^0_P(Y)=0$ and there are two exact sequences:
\[
0\to \KK\to H^0_{\dR}(\tilde Y)\to H^1_P(Y)\to 0
\]
and 
\[
0\to H^1_{\dR}(\tilde Y)\to H^2_P(Y)\to H^0_{\dR}(\tilde Y)\to
H^2_{\dR}(\tilde Y)\to
H^3_P(\tilde Y)\to H^1_{\dR}(\tilde Y)\to H^3_{\dR}(\tilde Y)\to \cdots
\]
where the maps $H^i_{\dR}(\tilde Y)\to H^{i+2}_{\dR}(\tilde Y)$ are given
by the cup product with the Chern class $\xi\in H^2_{\dR}(\tilde Y)$
of the hyperplane section 
(\emph{i.e.}, the first Chern class of $\calO_{\tilde
  Y}(1)$).\qed
\end{thm}

\begin{proof}[Proof of Theorem \ref{thm: main}]
The case when $j\leq 1$ is clear from the long exact sequence above.

Since the Picard group of $\tilde Y$ has rank 1, any two very ample
line bundles on $\tilde Y$ have a common power. It is thus
sufficient to consider the case
where the two ample line bundles in question are $\calL$ and
$\calL^m$.

Let $\xi(\calL)\in H^2_{dR}(\tilde Y)$
be the first Chern class of $\calL$, represented by a generic
hyperplane section with the embedding given by $\calL$. 
Then we have $\xi(\calL^m)=m\xi(\calL)$. Since the cup
product is linear and $\charac(\KK)=0$, the maps $H^i_{\dR}(\tilde
Y)\xrightarrow{\cup \xi} H^{i+2}_{\dR}(\tilde Y)$ and $H^i_{\dR}(\tilde
Y)\xrightarrow{\cup m\xi} H^{i+2}_{\dR}(\tilde Y)$ have the same
rank. Therefore $\dim_\KK(\ker(H^i_{dR}(\tilde Y)\to
H^{i+2}_{\dR}(\tilde Y)))$ and $\dim_\KK(\coker(H^i_{dR}(\tilde Y)\to
H^{i+2}_{\dR}(\tilde Y)))$ depend only on $\tilde Y$, but not on the
choice of the embedding (or equivalently, not on the choice of ample
line bundles $\calL$). When $j\geq 2$ we have
\[
\dim_\KK(H^j_P(Y)=\dim_\KK(\ker(H^{j-2}_{\dR}(\tilde Y)\to
H^{j}_{\dR}(\tilde Y))))+
\dim_\KK(\coker(H^{j-3}_{\dR}(\tilde Y)\to H^{j-1}_{\dR}(\tilde Y)))),
\]
hence the conclusion holds for $\dim_k(H^j_P(Y))$ when $j\geq 2$. 
\end{proof}

\begin{cor}
\label{indep for Ogus' q invariant}
Assume that the Picard group of $\tilde Y$ has rank 1. Then
$\lambda_{p,q}$ is independent of embeddings for all $q< n-f_Y$,
with $f_Y$ as in Remark \ref{rmk-Ogus-1}.
\end{cor}
\begin{proof}
Assume $q< n-f_Y$. Since $\Supp(H^{n-q}_I(R_n))\subseteq \{\frakm\}$,
\cite{L-Dmods} shows that $H^{n-q}_I(R_n)\cong
H^{n}_{\frakm}(R_n)^{\lambda_{0,q}}$ and $H^p_{\frakm}(H^{n-q}_I(R_n))=0$
for $p\geq 1$. Hence $\lambda_{p,q}=0$ for all $p\geq 1$ (and $q<
n-f_Y$).

Let $D(-)$ denote the Matlis dual. Then $D(H^{\ell}_I(R_n))\cong
\hat{R_n}^{\lambda_{0,n-\ell}}$ whenever $H^\ell_I(R_n)$ is
Artinian. On the other hand, \cite[Proposition 2.2,Theorem
  2.3]{Ogus-LCDAV} shows that, for $q<n-f_Y$,
\[
D(H^{n-q}_I(R_n))\cong H^q_P(\hat{X},\calO_{\hat{X}})\cong
\hat{R_n}\otimes H^q_P(Y)
\]
where $Y$ denotes the affine cone of $\tilde Y$ with vertex $P$ and $\hat{X}$
denotes the formal completion of $\Spec(\hat{R_n})$ along the subscheme
defined by $I$. This shows that
$\dim_\KK(H^q_P(Y))=\lambda_{0,q}$. Hence $\lambda_{0,q}$ depends only
on $\tilde{Y}$ by Theorem \ref{thm: main}.
\end{proof}
\begin{rmk}
  An alternative way to look at Corollary \ref{indep for Ogus' q
    invariant} arises through Proposition \ref{prop-complement-veronese}:
  for $q>f_Y$, the multiplicities of $H^n_\frakm(R_n)$ in $H^j_I(R_n)$ are
  exactly the Betti numbers $H^{n-1+j}(U)$ where $U$ is the affine
  complement of $Y$, because of the spectral sequence
  \eqref{eqn-Cech-deRham}.   By Proposition \ref{prop-complement-veronese}
  these do not change under Veronese maps.\schluss
\end{rmk}

\medskip

We now consider the effect of Serre's conditions $(S_t)$ in $R_n/I$ on
the Lyubeznik numbers.
\begin{rmk}
\label{Serre condition under faithfully flat}
Assume that ``$\tilde Y$ satisfies $(S_t)$ locally everywhere'', by which we mean that
each local ring $\calO_{\tilde Y,\tilde \fraky}$ of the projective
scheme $\tilde Y=\Proj(R_n/I)$
satisfies Serre condition $(S_t)$.

Let $Y$ be the cone $\Spec(R_n/I)$ as always and $P$ the vertex; then
the punctured cone $Y^\circ=Y\minus P$ is a bundle over $\tilde Y$. It
follows that every local ring of $Y^\circ$ also is $(S_t)$.  So for
each non-maximal prime ideal $\frakp$ of $R_n$ such that
$\dim((R_n/I)_\frakp)\geq t$, one has $\depth((R_n/I)_\frakp)\geq t$.

In
general, if $(A,\frakn)\to (A',\frakn')$ is a faithfully flat
morphism, then
  \[\depth(A')=\depth(A)+\depth(A'/\frakn A).\]
  If $A'$ is the strict Henselization $A^{sh}$ or the completion
  $\hat A$ of $A$, then $A'$ is faithfully flat over $A$. Therefore,
  \[
  \depth\left(
     \left(
     \left((R_n/I)_\frakp\right)^{\widehat{~}}
     \right)^{sh}
         \right)^{\widehat{~}}\geq  t.
  \] \schluss
\end{rmk}

\begin{lem}
\label{lem-S2ramification}
  If $\tilde Y$ is equi-dimensional and locally everywhere $(S_2)$ then
  the off-diagonal entries $\lambda_{i-1,i}$ vanish for
  $1< i<d:=\dim(\tilde Y)+1$, and $H^{n-1}_I(R_n)$ is Artinian and injective.
\end{lem}
\begin{proof}

  By Remark \ref{Serre condition under faithfully flat}, for
  each non-maximal prime ideal $\frakp$ of $R_n$ with $\dim((R_n/I)_\frakp)\geq 2$,
  we have $\depth\left(
     \left(
     ((R_n/I)_\frakp)^{\widehat{~}}
     \right)^{sh}
     \right)^{\widehat{~}}\geq  2$.
     Hence the punctured
  spectrum of this ring is connected by \cite{Hartshorne-CI+C}. The Second
  Vanishing Theorem implies that $H^{>\codim(P)-2}_I(R_n)_\frakp=0$ for each
  prime ideal $P$ such that $\dim((R_n/I)_\frakp)\geq 2$. Therefore the
  support dimension of $H^{i}_I(R_n)$ with $n-1>i>n-d$ is at most equal to
  $n-i-2$ and so $H^{i-1}_\frakm H^{n-i}_I(R)=0$ by Grothendieck's
  vanishing theorem. For $H^{n-1}_I(R_n)$, localization shows in
  conjunction with the Hartshorne--Lichtenbaum theorem that its
  support is at best at $P$. By Lyubzenik's work, it is hence Artinian
  and injective.
    \end{proof}

For the next three result we will use the following reduction.
\begin{lem}\label{lem-components}
  Let $\tilde Y$ be an equi-dimensional projective variety of
  dimension at least two. If the
  Lyubeznik numbers for the cones over all connected components of
  $\tilde Y$ are independent of the choice of the cone then the same
  is true for $\tilde Y$ itself.
\end{lem}
\begin{proof}
  Let $Y',Y''$ be two cones for $\tilde Y$ and let $\tilde Y=\tilde
  Y_1\sqcup \tilde Y_2$ be a disconnection. The resulting cones
  $Y'_1,Y''_1$ and $Y'_2,Y''_2$ satisfy: $Y'_1\cap Y'_2$ and
  $Y''_1\cap Y''_2$ both equal the origin. Let $I',I''$ be the
  defining ideals for $Y',Y''$ and denote the defining ideals of
  $Y'_1,Y'_2,Y''_1,Y''_2$ by $I'_1,I'_2\subseteq R_{n'}$ and
  $I''_1,I''_2\subseteq R_{n''}$
  respectively. All these ideals have dimension three or more.

  Then $H^q_{I'}(R_{n'})=H^q_{I'_1}(R_{n'})\oplus H^q_{I'_2}(R_{n'})$
  and $H^q_{I''}(R_{n''})=H^q_{I''_1}(R_{n''})\oplus
  H^q_{I''_2}(R_{n''})$ for all $q<n-1$ as follows from the
  Mayer--Vietoris sequence.

  It follows that, apart from $q=n,n-1$, the Lyubeznik numbers satisfy
  $\lambda_{p,q}(Y')= \lambda_{p,q}(Y'_1)+ \lambda_{p,q}(Y'_2)$ and
  $\lambda_{p,q}(Y'')= \lambda_{p,q}(Y''_1)+ \lambda_{p,q}(Y''_2)$. By
  the presumed embedding independence of $\Lambda(Y_1)$ and
  $\Lambda(Y_2)$, the same follows for $\Lambda(Y)$, except for
  columns $n,n-1$.

  In column $n$ all entries in all cases are zero by the
  Hartshorne--Lichtenbaum theorem. So is the diagonal entry
  $\lambda_{1,1}$ for all three ideals by equi-dimensionality. Thus,
  $\lambda_{0,1}(Y_i)=\lambda_{0,1}(Y'_i)+\lambda_{0,1}(Y''_i)+1$ for
  $i=1,2$ as follows from the Grothendieck spectral sequence (which
  implies that the alternating sum of all $\lambda_{p.q}$ is
  1). Therefore, all Lyubeznik numbers of $\tilde Y$ are embedding
  independent.
\end{proof}

\begin{thm}\label{thm-S2-P1}
  Let $\tilde Y$ be an equi-dimensional projective scheme of 
  dimension two, which
  \begin{enumerate}
  \item either satisfies locally everywhere Serre's condition $S_2$,
  \item or has Picard number one.
  \end{enumerate}
  Then the Lyubeznik numbers of all affine cones $Y$ over $\tilde Y$
  agree.
\end{thm}
\begin{proof}

  Let $Y$ be any cone over $\tilde Y$. It is a scheme of pure
  dimension 3, and thus by Remark \ref{rmk-L-argument}, the Lyubeznik
  table of $Y$ is
  \[
  \begin{pmatrix}
    \cdot&\lambda_{0,1}&\lambda_{0,2}&\cdot\\
    \cdot&\cdot&\lambda_{1,2}&\cdot\\
    \cdot&\cdot&\cdot&\lambda_{2,3}\\
    \cdot&\cdot&\cdot&\lambda_{3,3}\end{pmatrix}.
  \]

   By Lemma \ref{lem-components}, we can assume that $\tilde Y$ is
   connected. That assures that $\lambda_{0,1}$ is zero by the Second
   Vanishing Theorem \cite[Theorem 7.5]{Hartshorne-CDAV}. 

  If $\tilde Y$ is $(S_2)$ locally everywhere then by Lemma
  \ref{lem-S2ramification}, $H^{n-2}_I(R_n)$ has support dimension zero
  and is the top local cohomology module, and so
  $\lambda_{1,2}=0$. It follows from \cite{L-Dmods} that
  $H^{n-2}_I(R_n)$ is injective.  By Corollary \ref{cor-topLocalDeRham},
  the socle dimension $\lambda_{0,2}$ of this module is determined by
  the topology of $\tilde Y$.
  Finally, the convergence of the spectral sequence to $H^n_\frakm(R_n)$
  implies that $\lambda_{2,3}=\lambda_{0,2}$.

  Suppose now that $\tilde Y$ has Picard number one. Then by Lemma
  \ref{lem-consOfPic}, the $\lambda_{i,j}$ with $i>1$ are a function of
  $\tilde Y$ alone. The only possibly nonzero differentials are:
  \begin{itemize}
  \item on page two the morphism $E^{0,n-2}_2\to E^{2,n-3}_2$ and
    $E^{1,n-2}_2\to E^{3,n-3}_2$;
  \item on page three the morphism $E^{0,n-1}_3\to E^{3,n-3}_3$.
   \end{itemize}

  Convergence of the spectral sequence forces $E^{0,n-2}_2\to
  E^{2,n-3}_2$ to be an isomorphism\footnote{This isomorphism property
    holds for any ideal $I$ of dimension greater than two.}, and the
  maps $E^{1,n-2}_2\to E^{3,n-3}_2$ and $E^{0,n-1}_3\to
  E^{3,n-3}_3$ to be injective. Moreover, the cokernel of
  $E^{0,n-1}_3\to E^{3,n-3}_3$ must be one copy of
  $H^n_\frakm(R_n)$.

  Since all modules in $E^{p,q}_{\geq 2}$ are injective, socle
  dimensions are additive in short exact sequences. Thus,
  $\lambda_{0,2}=\lambda_{2,3}$, and
  $\lambda_{3,3}=\lambda_{1,2}+\lambda_{0,1}+1=\lambda_{1,2}+1$. This settles the
  claim for $\lambda_{0,2}$. 
 But 
  $\lambda_{3,3}$ is a function of $\tilde Y$  by
  \cite{Zhang-Compositio}, 
  and it follows that $\lambda_{1,2}$ is a function of $\tilde Y$ as well.
\end{proof}

\begin{thm} \label{thm-3fold}
Let $\tilde Y$ be a projective complex scheme that is of
equi-dimension three. Assume that every local ring $\calO_{\tilde Y
  ,\tilde y}$ satisfies $(S_2)$ and that the Picard group of $\tilde
Y$ has rank 1. Then $\Lambda(Y)$ is independent of the choice of the
cone $Y$ for $\tilde Y$.
\end{thm}

\begin{proof}
 By Lemma \ref{lem-components} we can assume that $\tilde Y$ is
 connected. This forces $\lambda_{0,1}(Y)=0$ for any cone $Y$ of
 $\tilde Y$ by the Second Vanishing Theorem \cite[Theorem 7.5]{Hartshorne-CDAV}. 
  
  Using the equi-dimensionality and the $(S_2)$-property, the Lyubeznik
  table is by Remark \ref{rmk-L-argument}  and Lemma \ref{lem-S2ramification}
 equal to 
  \[
\Lambda=  \begin{pmatrix}
    \cdot&\cdot&\lambda_{0,2}&\lambda_{0,3}&\cdot\\
    \cdot&\cdot&\cdot&\lambda_{1,3}&\cdot\\
    \cdot&\cdot&\cdot&\cdot&\lambda_{2,4}\\
    \cdot&\cdot&\cdot&\cdot&\lambda_{3,4}\\
    \cdot&\cdot&\cdot&\cdot&\lambda_{4,4}\end{pmatrix},
\]
  Moreover, $H^{n-2}_I(R_n)$ is supported only in the origin, hence
  injective. By Corollary \ref{cor-topLocalDeRham}, its socle
  dimension is the dimension of the top de Rham group of the affine cone
  complement.
  By Proposition \ref{prop-complement-veronese}, this dimension is
  well-defined. Thus, 
  $\lambda_{0,2}$ is a function of
  $\tilde Y$ alone, reflecting the de Rham group $H^{2n-2}(\AA^n_\KK\minus
  Y)$ independent of the choice of the cone.

  Convergence of the spectral sequence forces, similarly to the proof
  of Theorem \ref{thm-S2-P1}, that 
  $\lambda_{3,4}=\lambda_{1,3}+\lambda_{0,2}$ and that
  $\lambda_{2,4}=\lambda_{0,3}$. By the Picard number condition,
  $\lambda_{2,4}$ is the same for every cone, and hence so is
  $\lambda_{0,3}$. Since
  $\lambda_{0,2}$ is a function of $\tilde Y$, 
  and since $\lambda_{\geq 2,*}$ is independent of the embedding by the
  Picard number condition, the same is true for $\lambda_{1,3}$.
\end{proof}

\begin{thm}
\label{4-fold}
Let $\tilde Y$ be a projective complex scheme of equi-dimension
four. Assume that $\tilde Y$ is locally everywhere $(S_3)$, and that
the Picard group of $\tilde Y$ has rank 1. Then $\Lambda(Y)$ is
independent of the choice of the cone $Y$ for $\tilde Y$.
\end{thm}
\begin{proof}
   By Lemma \ref{lem-components} we can assume that $\tilde Y$ is
 connected. This forces $\lambda_{0,1}(Y)=0$ for any cone $Y$ of
 $\tilde Y$ by the Second Vanishing Theorem \cite[Theorem 7.5]{Hartshorne-CDAV}. 
  
  Write $\tilde Y=\Proj(R/I)$ where $R=\CC[x_1,\dots,x_n]$. Since
  $(S_3)$ implies $(S_2)$, Remark \ref{rmk-L-argument}
  and Lemma \ref{lem-S2ramification} assure that the Lyubeznik table of $R/I$ is
  \[
  \Lambda=\begin{pmatrix}
    \cdot&\cdot&\lambda_{0,2}&\lambda_{0,3}&\lambda_{0,4}&\cdot\\
    \cdot&\cdot&\cdot&\lambda_{1,3}&\lambda_{1,4}&\cdot\\
    \cdot&\cdot&\cdot&\cdot&\lambda_{2,4}&\lambda_{2,5}\\
    \cdot&\cdot&\cdot&\cdot&\cdot&\lambda_{3,5}\\
    \cdot&\cdot&\cdot&\cdot&\cdot&\lambda_{4,5}\\
    \cdot&\cdot&\cdot&\cdot&\cdot&\lambda_{5,5}\end{pmatrix}.
  \]
  Now take a prime $\frakp$ of height
  $n-2$ that contains $I$. Then
  $\depth((R_n/I)_\frakp)=3$ and so by
  \cite[Corollary 2.8]{Dao-Takagi},
  $(H^{(n-2)-3+1}_I(R_n))_\frakp=0$. Thus, $\dim(H^{n-4}_I(R_n))\le
  1$ and $\lambda_{2,4}=0$.

  Localizing at primes of height $n-1$ yields, with
  the result of Dao and Takagi \cite[Corollary 2.8]{Dao-Takagi}, that $H^{n-2}_I(R_n)$ and $H^{n-3}_I(R_n)$ are
  Artinian. It follows that
  $\lambda_{1,3}=0$, and $f_Y\le n-3$. 
  By Corollary \ref{indep for Ogus' q invariant}, since the Picard
  number is one, 
  $\lambda_{0,2}$ and $\lambda_{0,3}$ are independent of the embedding choice. 
  
  Convergence of the spectral sequence to $H^n_\frakm(R)$ forces that
  \[
 \lambda_{0,4}=\lambda_{2,5}\qquad {\rm and }\qquad
  \lambda_{1,4}=\lambda_{3,5}-\lambda_{0,3} \qquad {\rm (and }\qquad  \lambda_{0,2}=\lambda_{4,5}).
  \]
  As the Picard number is one, the $\lambda_{i,j}$ are independent of
  embeddings for all $i\geq 2$ and all $j$. This then fixes all $
  \lambda_{p,q}$.   
\end{proof}

\begin{proof}[{Proof of Theorem \ref{thm-toric-3fold}}]
  Toric projective varieties are connected and locally the spectra of
  semigroup rings to saturated semigroups. They are hence normal, and
  by Hochster's theorem Cohen--Macaulay, \cite{Hochster}. The
  coordinate ring $R_n/I$ of the cone $Y$ thus has a Lyubeznik table
  as in the proof of Theorem \ref{thm-3fold}. Moreover,
  $H^{n-2}_I(R_n)$ is Artinian.

  Additional vanishings are due to the $(S_3)$-condition. As in
  the proof of Theorem \ref{4-fold}, localization at a prime of $R_n$
  of height $n-1$ shows with \cite[Thm.~2.8]{Dao-Takagi} that the
  support of $H^{n-3}_I(R_n)$ is zero-dimensional, hence
  $\lambda_{1,3}=0$.

  At this point, let us assume that $\tilde Y$ is not a hypersurface,
  and hence of codimension two or more.

  If $\lambda_{0,2}$ is nonzero, it is therefore the dimension of
  $H^n_\dR(H^{n-2}_I(R_n))=H^{2n-3}(U;\CC)$ where $U$ is the affine
  complement of $Y$. By the spectral sequence argument in the proof of
  Lemma \ref{lem-topLocalDeRham}, it also equals the dimension of the top
  cohomology group $H^{2n-4}(\PP U;\CC)$ of the projective complement $\PP U$.

  The long exact sequence \eqref{eqn-les-Iversen}
  takes the form 
  \[
  H^{2n-4}_{\tilde Y}(\PP^{n-1};\CC)\to H^{2n-4}(\PP^{n-1};\CC)\to
  H^{2n-4}(\PP U;\CC)\to H^{2n-3}_{\tilde Y}(\PP^{n-1};\CC)\to \cdots
  \]
  and by \cite[V.6.6]{Iversen} $H^{2n-3}_{\tilde Y}(\PP^{n-1};\CC)$ is
  dual to $H^1_c(\tilde Y;\CC)=H^1(\tilde Y;\CC)$. But projective
  toric varieties (or more generally toric varieties to a fan with a
  full-dimensional cone) are simply connected by
  \cite[3.2]{Fulton}. So $H^1(\tilde Y;\CC)$ and $H^{2n-3}_{\tilde
    Y}(\PP^{n-1};\CC)$ are zero.

  The morphism $ H^{2n-4}_{\tilde Y}(\PP^{n-1};\CC)\to
  H^{2n-4}(\PP^{n-1};\CC)$ is via Alexander and Poincar\'e duality
  dual to the (injective) restriction morphism $H^2(\PP^{n-1};\CC)\to
  H^2(\tilde Y;\CC)$, hence itself surjective. It follows that
  $H^{2n-4}(\PP U;\CC)=H^{2n-3}(U;\CC)=0$, and hence
  $\lambda_{0,2}$ and $H^{n-2}_I(R_n)$ are both zero.

  It now follows that actually the Artinian module $H^{n-3}_I(R_n)$ is the top
  local cohomology group of $I$, and $H^{2n-5}(\PP U;\CC)$ is the top
  cohomology group of $\PP U$. Repeating the above computations, we
  now have a long exact sequence
   \[
  0= H^{2n-5}(\PP^{n-1};\CC)\to
  H^{2n-5}(\PP U;\CC)\to H^{2n-4}_{\tilde Y}(\PP^{n-1};\CC)\to
  H^{2n-4}(\PP^{n-1};\CC)\to \cdots
  \]
  in which the arrow $H^{2n-4}_{\tilde Y}(\PP^{n-1};\CC)\to
  H^{2n-4}(\PP^{n-1};\CC)$ is dual to the (injective) morphism
  $H^2(\PP^{n-1};\CC)\to H^2(\tilde Y;\CC)$, and where  $H^{2n-5}(\PP
  U;\CC)=H^{2n-4}(U;\CC)$ is a vector space of dimension
  $\lambda_{0,3}$.

  For projective toric varieties (and more generally, when all cones
  of the fan are top-dimensional), the Picard group of $\tilde Y$ is
  isomorphic to $H^2(\tilde Y;\CC)$,
  \cite[Thm.~12.3.2]{CoxLittleSchenck}. The long exact sequence above
  thus shows the equation $\lambda_{0,3}=\tilde p$.

  Finally, by convergence of the spectral sequence, $\lambda_{4,4}=1$ and
  $\lambda_{3,4}=\lambda_{0,2}=0$ and $\lambda_{2,4}=\lambda_{0,3}$.

  This settles the problem for all embeddings in which $\tilde Y$ is
  not a hypersurface. If in some embedding $\tilde Y$ happens to be a
  hypersurface, necessarily in $\PP^4_\CC$, its Lyubeznik table is
  trivial for this embedding, simply for lack of higher local
  cohomology. On the other hand, \cite[Exp.~XII, Cor 3.7]{SGA2}
  asserts that the Picard group of $\tilde Y$ is then cyclic, equal to
  that of $\PP^4_\CC$. Thus, $\tilde p$ is zero and we see that all
  Lyubeznik tables of $\tilde Y$ agree.
\end{proof}

\section*{Acknowledgements}
\thanks{We would like to thank Winfried Bruns, Boris
  Pasquier and Alex Dimca for helpful conversations.}

\bibliographystyle{amsalpha}
\bibliography{veronese.bib}
\end{document}